\RequirePackage{fix-cm}
\documentclass[smallextended]{svjour3}       
\smartqed  
\usepackage{graphicx}
\usepackage{amsmath}
\usepackage{algorithm}
\usepackage{algpseudocode}
\usepackage{dcolumn,multirow}
\usepackage{here}
\usepackage{comment}
\usepackage{cite}
\usepackage{amssymb}
\usepackage{multirow}
\usepackage{mathtools}
\mathtoolsset{showonlyrefs}
\usepackage{color}
\usepackage{graphicx}
\usepackage{url}
\newcommand{\rast}{\rho_{\ast}}
\newcommand{\dimh}{s}
\newcommand{\ol}{\overline}
\newcommand{\mua}{\mu}
\def\tims{time(s)}
\def\outite{$\sharp$ite(Alg\ref{alg4})}
\def\innite{$\sharp$ite(Inner)}
\def\resKKT{$\|\Phi_0^1\|$}
\def\resKKTMT{$\|\Xi_0^I\|$}
\def\fail{success(\%)}
\newcommand{\dprime}{\prime\prime}
\newcommand{\fr}{\frac{1}{2}}

\newcommand{\tx}{\tilde{x}}
\newcommand{\tY}{\widetilde{Y}}
\newcommand{\tz}{\tilde{z}}
\newcommand{\tl}{\widehat{t}_{\ell}}
\newcommand{\Gi}{\mathcal{G}_i}
\newcommand{\twc}{\tw_{\rm c}}
\newcommand{\txt}{\tx_{\rm t}}
\newcommand{\txc}{\tx_{\rm c}}

\newcommand{\ellem}{l}
\newcommand{\mO}{O}
\newcommand{\twt}{\tw_{\rm t}}
\newcommand{\Uellast}{S_{\ell}^{i}}
\newcommand{\zli}{Z_{\ell}^i}

\newcommand{\Ql}{Q_{\ell}}
\newcommand{\Uls}{\mathcal{U}_{\ell}^i}
\newcommand{\UQl}{\mathcal{U}_{\Ql}^i}
\newcommand{\dd}{d}
\newcommand{\muk}{\mu_k}
\newcommand{\Kl}{K_{\ell}}

\newcommand{\R}{\mathbb{R}}
\newcommand{\HKM}{H.K.M.}
\newcommand{\Gl}{G_{\ell}}

\newcommand{\Gil}{\Gi(\txl)}
\newcommand{\hGl}{\widehat{G}_{\ell}}
\newcommand{\hYl}{\widehat{Y}_{\ell}}
\renewcommand{\P}{\mathcal{P}}
\newcommand{\y}{y^{(i)}}
\newcommand{\rTheta}{{\rm\Theta}}
\newcommand{\W}{\mathcal{W}}
\renewcommand{\l}{\ell}

\newcommand{\tw}{\tilde{w}}
\newcommand{\tmu}{\tilde{\mu}}
\newcommand{\twl}{\tilde{w}^{\ell}}

\newcommand{\txl}{\tilde{x}^{\ell}}
\newcommand{\tmul}{\tilde{\mu}_{\ell}}
\newcommand{\tYl}{\widetilde{Y}_{\ell}}

\newcommand{\tP}{\widetilde{P}}
\newcommand{\tPl}{\tP_{\ell}}

\newcommand{\kfr}{k+\fr}

\renewcommand{\epsilon}{\varepsilon}
	
\newcommand{\Wl}{W_{\ell}}

\newcommand{\rO}{{\rm O}}
\newcommand{\ro}{{\rm o}}

\renewcommand{\phi}{\varphi}

\newcommand{\Deltac}{\Delta_{\rm c}}
\newcommand{\Deltap}{\Delta_{\rm t}}

\newcommand{\G}{\mathcal{G}}
\newcommand{\hGi}{\widehat{\mathcal{G}}_i}

\renewcommand{\refname}{References}

\def\ru[#1][#2]{{#1}^{#2}}
\def\rl[#1][#2]{{#1}_{#2}}
\renewcommand{\hat}{\widehat}

\newcommand{\N}{\mathcal{N}}

\newcommand{\hY}{\widehat{Y}}
\newcommand{\mS}{\mathbb{S}}
\newcommand{\hU}{U}
\newcommand{\olmu}{\ol{\mu}}
\newcommand{\olw}{w}
\newcommand{\hG}{\widehat{G}}

\newcommand{\Wcal}{\mathcal{W}}

\newcommand{\cref}{\ref}
\spnewtheorem{assumption}{Assumption}{\bf}{\it}
\spnewtheorem{PropA}{Proposition}[section]{\bf}{\it}
\journalname{}
\begin{document}

\title{Local convergence of primal-dual interior point methods for nonlinear semidefinite optimization using the Monteiro-Tsuchiya family of search directions  
\thanks{
This research was supported in part by Grant-in-Aid for Young Scientists 20K19748 and 
Grant-in-Aid for Scientific Research (C)20H04145 from JSPS KAKENHI.}
}

\titlerunning{Interior point method for NSDPs based on the MT family}        

\author{Takayuki Okuno}


\institute{T. Okuno \at
Faculty of Science and Technology Department of Science and Technology, Seikei University, Tokyo, Japan\\
              Center for Advanced Intelligence Project, RIKEN, Tokyo, Japan\\ \email{takayuki-okuno@st.seikei.ac.jp}           
}

\date{Received: date / Accepted: date}

\maketitle

\begin{abstract}
The recent advance of algorithms for nonlinear semidefinite optimization problems (NSDPs) is remarkable.
Yamashita et al. first proposed a primal-dual interior point method (PDIPM) for solving NSDPs using the family of Monteiro-Zhang (MZ) search directions.
Since then, various kinds of PDIPMs have been proposed for NSDPs, but, as far as we know, all of them are based on the MZ family.
In this paper, we present a PDIPM equipped with the family of Monteiro-Tsuchiya (MT) directions, which were originally devised for solving linear semidefinite optimization problems as were the MZ family. We further prove local superlinear convergence to a Karush-Kuhn-Tucker point of the NSDP in the presence of certain general assumptions on scaling matrices, which are used in producing the MT search directions. Finally, we conduct numerical experiments to compare the efficiency among members of the MT family.
\end{abstract}

\section{Introduction}
In this paper, we consider the following nonlinear semidefinite optimization problem:
\begin{align}\label{al:nsdp}
\begin{array}{rcl}
\displaystyle\mathop{\rm Minimize}& &f(x)\\
                              \mbox{subject~to}& &G(x)\in \mS^m_+,\\
                                               & &h(x) = 0,
\end{array}
\end{align}
where
$f:\R^n\to \R$, $G:\R^n\to \mS^m$, and $h:\R^n\to \R^{\dimh}$
are twice continuously differentiable functions. Also, $\mS^m$ denotes the set of real 
$m\times m$ symmetric matrices
and $\mS^m_{++}$($\mS^m_+$) stands for the set of 
$m\times m$ real symmetric positive definite (positive semidefinite) matrices.
Throughout the paper, we
often refer to problem\,\eqref{al:nsdp} as NSDP.

When $G$ takes a diagonal matrix form, 
$G(x)\in \mS^m_+$ means all the diagonal elements of $G(x)$ are nonnegative, and hence the NSDP becomes a standard nonlinear optimization problem.
When all the functions are affine with respect to $x$,
the NSDP reduces to the linear semidefinite optimization problem (LSDP). 
The LSDP has been extensively studied on the aspects of theory, algorithms, and applications so far, as they are very powerful tools in various fields, one of which being combinatorial optimization. 
See \cite{wolkowicz2012handbook,vandenberghe1996semidefinite} for a comprehensive survey on the LSDP. 
On the other hand, 
studies on the NSDP itself have also advanced significantly in the 2000s.
The NSDP arises in a wide variety of  applications
such as 
structural optimization\,\cite{kovcvara2004solving}, control\,\cite{scherer1995multiobjective,kovcvara2005nonlinear,hoi2003nonlinear,leibfritz2006reduced}, statistics\,\cite{qi2006quadratically}, finance\,\cite{konno2003cutting,leibfritz2009successive}, and so on.
Elaborate theoretical results on optimality for the NSDP have been also developed. For example, the Karush-Kuhn-Tucker (KKT) conditions and the second-order conditions for the NSDP were studied in detail by Shapiro\,\cite{shapiro1997first} and Forsgren\,\cite{forsgren2000optimality}. 
Further examples include the strong second-order conditions by Sun\,\cite{sun2006strong}, sequential optimality conditions by Andreani et al\,\cite{andreani2018optimality}, and the local duality by Qi\,\cite{qi2009local}.
Along with such theoretical results, various algorithms have been proposed for solving the NSDP, for example, augmented Lagrangian methods\,\cite{kovcvara2004solving,sun2006properties,sun2008rate,andreani2018optimality,fukuda2018exact,huang2006approximate},
sequential linear semidefinite programming methods\,\cite{kanzow2005successive}, sequential quadratic semidefinite programming methods\,\cite{correa2004global,freund2007nonlinear,zhao2016superlinear,yamakawa2019stabilized,zhao2020line},
sequential quadratically constrained quadratic semidefinite programming methods\,\cite{auslender2013extended},
interior point-type methods\,\cite{jarre2000interior,leibfritz2002interior,yamashita2012local,yamashita2012primal,yamakawa1,yamakawa2,yamashita2020primal,kato2015interior,okuno2020interior,okuno2018primal}, and so forth. 

Let us review existing studies on primal-dual interior point methods (PDIPMs) for the NSDP in more detail.
Similar to the path-following PDIPMs for the LSDP, the fundamental framework of the existing PDIPMs for the NSDP is to approach a KKT triplet of the NSDP, by approximately computing perturbed KKT triplets, and by driving a perturbation parameter to zero.
These perturbed KKT triplets are called Barrier KKT (BKKT) triplets, and
the perturbation parameter is called a barrier parameter.
To the best of the author's knowledge, the first PDIPM for the NSDP was presented by Yamashita, Yabe, and Harada\,\cite{yamashita2012primal}, who showed its global convergence to a KKT triplet of the NSDP under some assumptions. Its local convergence property was analyzed by Yamashita and Yabe in \cite{yamashita2012local}, who proved its superlinear convergence for the Alizadeh-Haeberly-Overton (AHO) directions, and two-step superlinear convergence for the Nesterov-Todd (NT) and Helmberg-Rendle-Vanderbei-Wolkowicz/Kojima-Shindoh-Hara/Monteiro (\HKM) directions. 
Those directions are direct extensions of the AHO\,\cite{alizadeh1998primal}, NT\,\cite{nesterov1998primal,todd1998nesterov}, and \HKM directions\,\cite{
helmberg1996interior,kojima1997interior,monteiro1997primal} for the LSDP.
These advances spurred a great deal of interest in PDIPMs for the NSDP.
Kato et al\,\cite{kato2015interior} presented the primal-dual quadratic penalty function as a merit function for the global convergence. Yamakawa and Yamashita\,\cite{yamakawa1} introduced the shifted BKKT conditions as an alternative to the BKKT conditions for the NSDP, and in \cite{yamakawa2} they also showed the two-step superlinear convergence of a PDIPM which uses two kinds of search directions produced by solving two scaled Newton equations sharing the same Jacobian.
In \cite{yamashita2020primal}, Yamashita, Yabe, and Harada integrated a trust-region technique into a PDIPM.
Okuno and Fukushima\,\cite{okuno2020interior,okuno2018primal} considered special NSDPs which posses an infinite number of convex inequality constraints, and proposed a PDIPM coupled with sequential quadratic programming methods.

  At the time of writing this work, all the existing PDIPMs for the NSDP are based on the Monteriro-Zhang (MZ) family of search directions, which was introduced by 
  Zhang\,\cite{zhang1998extending} in the framework of the LSDP to analyze PDIPMs using different search directions such as {\HKM} and AHO directions in a unified manner.
 The MZ family of search directions is briefly described: from the KKT conditions, 
   the semidefinite complementarity condition $G(x)Y=0$ with $G(x),Y\in \mS^m_+$ holds, where $Y$ is a dual matrix variable in $\mS^m$.
   Perturbing the right-hand side of this equation in terms of a parameter $\mu>0$, we have $G(x)Y=\mu I$ together with $G(x),Y\in \mS^m_{++}$. 
   Symmetrizing the former equation,
   we obtain $G(x)Y+YG(x)=2\mu I$, where $I$ denotes the identity matrix in $\R^{m\times m}$. Moreover, scaling $G(x)$ and $Y$ to $PG(x)P^{\top}$ and
    $P^{-\top}YP^{-1}$ with a scaling matrix $P\in \R^{m\times m}$ being nonsingular, respectively, we obtain
    $$PG(x)YP^{-1} + P^{-\top}YG(x)P = 2\mu I.$$
    A direction in the MZ family is a solution of Newton's equation to this equation along with other equations derived from the KKT conditions.
    When $P:=I$, the corresponding direction is the AHO direction. See Section\,\ref{subsec:MZ} for more details.

The purpose of the present paper is to develop PDIPMs for the NSDP using a different family of search directions,
named the Monteiro-Tsuchiya (MT) family, whose members are obtained by applying the Newton method to 
$$(PG(x)P^{\top})^{\fr}P^{-\top}YP^{-1}(PG(x)P^{\top})^{\fr}= \mu I.$$
The MT family was initiated by Monteiro and Tsuchiya\,\cite{monteiro1999polynomial} for solving the LSDP.
When $P=I$, they proved the polynomial convergence of the proposed PDIPMs with the short- and semilong-step strategies. 
This shows superiority to the AHO direction in the MZ, for which only the short-step PDIPM was shown to be polynomially convergent by Monteiro\,\cite{monteiro1998polynomial}.
Moreover, they showed the polynomial convergence of short- and semilong step PDIPMs for the whole MT family, and proved the same for the long-step PDIPM for the MT* family, a certain sub-family of the MT.  
The reader may be referred to \cite{monteiro1999polynomial,mj1999study,monteiro1999implementation,kakihara2014curvature,lu2005error} for relevant works on the MT family for the LSDP.

A motivation of considering the MT family in the context of the NSDP is that 
any member of the MT family is a decent direction of a certain merit function 
to measure the barrier KKT optimality. This indicates huge potential for the globally convergent PDIPM for the whole MT family by using this merit function, although 
we will focus on the local convergence in the current paper. On the other hand,
such a function has not been discovered so far as for the MZ family. The existing global convergence analyses of the PDIPMs using the MZ family for the NSDP are established only for search directions corresponding to scaling matrices $P$ such that scaled $G(x)$ and $Y$ commute\,\cite{okuno2020interior,yamakawa1,yamashita2020primal}.
This fact may imply a theoretical advantage of the MT over the MZ. 
For detailed discussion, see Section\,\ref{sec:MT}.

The PDIPM which will be presented mainly consists of two steps: The first is a tangential step, in which we move towards the set of KKT triplets of the NSDP along a direction tangential to
a central path that is a set of points satisfying the barrier KKT (BKKT) conditions.
Based on the step size with which we move in this direction, we determine the next barrier parameter. 
The second step is a centering step, in which we move on to the next iterate improving the deviation from the set of BKKT triplets with the updated barrier parameter.
These two steps are constructed from MT search directions, and their corresponding step sizes are adjusted in order to ensure that the primal-dual iterates remain in
the interior of the semidefinite cone constraints.
In this paper, the local superlinear convergence of such a PDIPM is of our interest, and we leave the global convergence to future studies.
The local convergence analysis is specialized to the MT family of search directions and different from the existing analysis of the PDIPMs\,\cite{yamashita2012local,yamakawa2,okuno2018primal,okuno2020interior} using the MZ search directions. For details, see Remark\,\cref{rem:1}.

The remainder of the paper is organized as follows:
In Section\,\cref{sec:2}, we introduce the MT family for the NSDP together with its key members.  
In Section\,\cref{sec:3}, we present a PDIPM using the MT family of search directions.
In Section\,\cref{sec:local}, we show local superlinear convergence of the presented PDIPM based on the MT family.	
In Section\,\cref{sec:num}, we conduct some numerical experiments to compare efficiency of members of the MT family. 
Finally, in Section\,\cref{sec:con}, we conclude this paper with some remarks.
The proofs of some lemmas and propositions in Section\,\cref{sec:local} are provided in Appendix.
\section*{Notations and terminologies}
We denote the identity and zero matrices in $\R^{m\times m}$ by $I$ and $O$, respectively.
For $A\in \R^{m\times m}$, we define ${\rm Sym}(A):=(A+A^{\top})/2$ and $\|A\|_{\rm F}:=\sqrt{{\rm trace}(A^{\top}A)}$.
For $X,Y\in \mS^m$, we define the inner product $X\bullet Y$ by $X\bullet Y:= {\rm trace}(XY)$. We also define the linear operator $\mathcal{L}_X:\mS^m\to \mS^m$ by 
$\mathcal{L}_X(Y):=XY+YX$. We denote the smallest eigenvalue of $X\in \mS^m$ by $\lambda_{\rm min}(X)$.
For $X\in \mS^m_{+}$ and $r>0$, we denote by $X^{\frac{1}{r}}$ the unique solution $U\in \mS^m_+$ of $U^r=X$. 
For a function $g:\R^n\to \R$, we denote by $\nabla g(x)$ or $\nabla_x g(x)$ the gradient of $g$, namely, $\nabla g(x):=(\frac{\partial g(x)}{\partial x_1},\ldots,
\frac{\partial g(x)}{\partial x_n}
)^{\top}\in \R^n$ and, also denote by $\nabla^2_{xx}g(x)$ the hessian of $g$, namely, $\nabla^2_{xx}g(x)=(\frac{\partial^2 g(x)}{\partial x_i\partial x_j})_{1\le i,j\le n}\in \R^{n\times n}$.
For sequences $\{a_{k}\}, \{b_k\}\subseteq \R$, 
we write $a_k={\rm O}(b_k)$ if there exists some $M>0$ such that 
$|a_k|\le M|b_k|$ for all $k$ sufficiently large, and write $a_k={\rm o}(b_k)$ if there exists some negative sequence $\{\alpha_k\}\subseteq \R$ such that 
$\lim_{k\to \infty}\alpha_k=0$ and 
$|a_k|\le \alpha_k|b_k|$ for all $k$ sufficiently large. We also say $a_k={\rm \Theta}(b_k)$ if there exist $M_1,M_2>0$ such that
$M_1|b_k|\le |a_k|\le M_2|b_k|$ for all $k$ sufficiently large.   

For $i=1,2,\ldots,n$, we write
$$\Gi(x):=\frac{\partial G(x)}{\partial x_i}.$$
We also denote $\R_{++}:=\{a\in \R\mid a>0\}$ and 
\begin{align*}
  \mathcal{W}:=\R^n\times \mS^m\times \R^{\dimh},\
  \mathcal{W}_{++}:=\{
(x,Y,z)\in \mathcal{W}\mid G(x)\in \mS^m_{++}, Y\in \mS^m_{++}
  \}.
\end{align*}
Additionally, letting $\W_+$ be the set obtained by replacing $\mS^m_{++}$ with $\mS^m_+$ in $\W_{++}$ and $w:=(x,Y,z)\in \W$, we write $\|w\|:=\sqrt{\|x\|_2^2+\|Y\|_{\rm F}^2+\|z\|_2^2}$, where $\|\cdot\|_2$ represents the Euclidean norm. 

\section{Preliminaries}\label{sec:2}
\subsection{The KKT and BKKT optimality conditions}
We begin by introducing the Karush-Kuhn-Tucker (KKT) conditions for NSDP\,\eqref{al:nsdp}.
\begin{definition}
We say that the the Karush-Kuhn-Tucker (KKT) conditions for NSDP\,\eqref{al:nsdp}
hold at $x\in \R^n$ if there exist a Lagrange multiplier matrix $Y\in \mS^m$ and a vector $z\in \R^{\dimh}$ such that 
\begin{subequations}
\begin{align}
&\nabla_xL(w)=\nabla f(x) - \mathcal{J}G(x)^{\ast}Y+\nabla h(x)z=0,
\label{al:kkt1}\\
&G(x)\bullet Y = 0,\ G(x)\in \mS^m_+,\ Y\in \mS^m_+,\label{al:comp}\\
&h(x) = 0, 
\label{al:kkt3}
\end{align}
\end{subequations}
where
$w:=(x,Y,z)\in \mathcal{W}$, $\mathcal{J}G(x)^{\ast}Y:=(\G_1(x)\bullet Y,\G_2(x)\bullet Y,\ldots,\G_n(x)\bullet Y)^{\top}$ and $L:\mathcal{W}\to \R$ denotes the
Lagrangian function for the NSDP, that is, 
\begin{equation}
L(w):=f(x) - G(x)\bullet Y + h(x)^{\top}z \label{eq:lo}
\end{equation}
for any $w\in\mathcal{W}$.
Particularly, we call a triplet $w=(x,Y,z)$ satisfying the KKT conditions a KKT triplet of NSDP\,\eqref{al:nsdp}. 
\end{definition}
Below, we define the Mangasarian-Fromovitz constraint qualification (MFCQ), under which 
the KKT conditions are necessary optimality conditions of the NSDP. 
\begin{definition}
Let ${x}\in \R^n$ be a feasible point of NSDP\,\eqref{al:nsdp}.
We say that
the Mangasarian-Fromovitz constraint qualification (MFCQ) holds at ${x}$
if 
$\nabla h({x})$ is of full column rank and 
there exists a vector $d\in \R^n$ such that $\nabla h(x)^{\top}d=0$ and 
$G({x})+\mathcal{J}G({x})d\in \mS^m_{++}$, 
where $\mathcal{J}G({x})d:=\sum_{i=1}^nd_i\Gi({x})$.
\end{definition}
According to \cite[Corollary~2.101]{bonnans2013perturbation}, the MFCQ is equivalent to Robinson's constraint qualification that is obtained by replacing $G({x})+\mathcal{J}G({x})d\in \mS^m_{++}$ with $0\in {\rm int}(G({x})+\mathcal{J}G({x})d-\mS^m_+)$ in the MFCQ.  
\begin{remark}
Let ${x}\in \R^n$ be a local optimum of NSDP\,\eqref{al:nsdp}.
Under the MFCQ, the KKT conditions hold at ${x}$. 
Conversely, if $f$ is convex, $h$ is affine, and $G$ is matrix convex in the sense of \cite[Section~5.3.2]{bonnans2013perturbation}, ${x}$ satisfying the KKT conditions is a global optimum of \eqref{al:nsdp}.
\end{remark}
The semidefinite complementarity condition\,\eqref{al:comp} 
{can be represented equivalently as 
\begin{equation}
G(x)Y=O,\ G(x)\in \mS^m_+,\ Y\in \mS^m_+.\label{eq:GY}
\end{equation}
We perturb \eqref{eq:GY} in terms of a parameter $\mu>0$ as 
\begin{align}
&G(x)Y=\mu I,\ G(x)\in \mS^m_{++},\ Y\in \mS^m_{++}.\label{eq:bGY1}
\end{align}
\begin{definition}
Let $\mu >0$.
We say that the barrier KKT (BKKT) conditions 
hold at $x$ for NSDP\,\eqref{al:nsdp} if there exists  
$(Y,z)\in \mS^m\times \R^{s}$ which satisfies \eqref{al:kkt1}, \eqref{al:kkt3}, and \eqref{eq:bGY1}.
In particular, we often call $\mu$ and $w=(x,Y,z)$ satisfying the BKKT conditions 
a barrier parameter and a BKKT triplet, respectively. 
\end{definition}
There exist equivalent representations for \eqref{eq:bGY1} for all $\mu>0$, among which
the following two expressions are given for later use:  
\begin{equation}
{\rm Sym}\left(G(x)Y\right)=\mu I,\ G(x)\in \mS^m_{++},\ Y\in \mS^m_{++}\label{eq:aho}
\end{equation}
and 
\begin{equation}
G(x)^{\frac{1}{2}}YG(x)^{\frac{1}{2}}=\mu I,\ G(x)\in \mS^m_{++},\ Y\in \mS^m_{++}.\label{eq:MT}
\end{equation} 
}
\subsection{The MZ family of  search directions for the NSDP}\label{subsec:MZ}
When we are concerned with standard nonlinear optimization, it is common to apply the Newton method to the equation system consisting of the BKKT conditions\,\eqref{al:kkt1}, \eqref{al:kkt3}, and the first equation of \eqref{eq:bGY1} to generate search directions. However, 
it does not work for the NSDP, because the number of equations is $n+m^2+s$ in the system whereas that of the variables $(x,Y,z)\in \W$ is $n+\frac{m(m+1)}{2}+s$, which is less than the former when $m\ge 2$.

A remedy for this issue is to use the BKKT conditions replacing \eqref{eq:bGY1} with \eqref{eq:aho}, which results in the search direction called the AHO direction. 
This was originally presented by Alizadeh, Haeberly, and Overton\,\cite{alizadeh1998primal} for the LSDP.
The Monteiro-Zhang (MZ) family is the set of scaled AHO directions obtained by applying the Newton method to the BKKT conditions, scaling \eqref{eq:aho} further with a nonsingular matri $P\in \R^{m\times m}$ as 
\begin{equation}
  {\rm Sym}\left(\hG(x)\hY\right)=\mO,\ \hG(x)\in \mS^m_+,\ \hY\in \mS^m_+,\label{eq:saho}
\end{equation}
where
\begin{equation}
\widehat{G}(x):=PG(x)P^{\top},\ \widehat{Y}:=P^{-\top}YP^{-1}.\label{eq:scalegxy}
\end{equation}
A search direction $(\Delta x,\Delta Y)\in \R^n\times \mS^m$ belonging to the MZ family is a solution to 
\begin{equation}
  {\rm Sym}\left(\hG(x)P^{-\top}\Delta YP^{-1}+ \sum_{i=1}^n\Delta x_i\widehat{\Gi}(x)\hY\right)=\mu I - {\rm Sym}\left(\hG(x)\hY\right),\label{eq:mz}
\end{equation}
where 
\begin{equation}
\widehat{\G}_i(x):=P\Gi(x)P^{\top}\ (i=1,2,\ldots,n).\label{eq:sGi}
\end{equation}
The MZ family was originally developed to analyze different search directions in a unified manner for the LSDP\,\cite{zhang1998extending,monteiro1998unified}. 
Indeed, the NT and {\HKM} directions are obtained as members of the MZ family by setting 
$P=W^{-\fr}, G(x)^{-\fr}$, respectively, where
\begin{equation}
W:=G(x)^{\fr}(G(x)^{\fr}YG(x)^{\fr})^{-\fr}G(x)^{\fr}.\footnote{We have the identity $W=Y^{-\fr}(Y^{\fr}G(x)Y^{\fr})^{\fr}Y^{-\fr}$.}\label{eq:W}
\end{equation}
In the context of the NSDP, Yamashita et al.\,\cite{yamashita2012primal,yamashita2012local} were 
the first to study PDIPMs using the MZ family.

\subsection{The MT family of search directions for the NSDP}\label{sec:MT}
In this section, we present the Monteiro-Tsuchiya (MT) family of search directions that are tailored to NSDP\,\eqref{al:nsdp} and introduce its important members.
{Scaling $G(x)$ and $Y$ as \eqref{eq:scalegxy} in condition\,\eqref{eq:MT}, we have
\begin{equation}
\hat{G}(x)^{\frac{1}{2}}\hat{Y}\hat{G}(x)^{\frac{1}{2}}=\mu I,\ \hG(x)\in \mS^m_{++},\ \hY\in \mS^m_{++}.\label{eq:sbMT}
\end{equation}
A member of the MT family is produced by applying the Newton method to 
the equation system comprising 
\eqref{al:kkt1}, \eqref{al:kkt3}, and  the first equation of \eqref{eq:sbMT}.
Applying the first-order Taylor's expansion to the first equation in \eqref{eq:sbMT} yields the following linear equation in $(\Delta x,\Delta {Y})\in \R^n\times \mS^m$:
\begin{equation}
\hat{G}(x)^{\frac{1}{2}}\left(P^{-\top}\Delta {Y}P^{-1}\right)\hat{G}(x)^{\frac{1}{2}}
+ \hU\hat{Y}\hat{G}(x)^{\frac{1}{2}}
+ \hat{G}(x)^{\frac{1}{2}}\hat{Y}\hU
=\mu I - \hat{G}(x)^{\frac{1}{2}}\hat{Y}\hat{G}(x)^{\frac{1}{2}}, \label{eq:mtlinear}
\end{equation}
where $\hU\in \mS^m$ denotes the G\^ateaux differential for $\widehat{G}(x)^{\frac{1}{2}}$
in the direction $\Delta x$.
In a manner similar to \cite[Lemma~2.2]{monteiro1999polynomial}, 
we can show that the matrix $\hU$ solves the following linear equation parameterized by $\Delta x$:
\begin{equation}
  \left(\hU\hat{G}(x)^{\frac{1}{2}}+\hat{G}(x)^{\frac{1}{2}}\hU=\right)
  \mathcal{L}_{\hG(x)^{\fr}}(U)
= \sum_{i=1}^n\Delta x_i\widehat{\Gi}(x),\label{al:1214-1}
\end{equation}
where $\widehat{\Gi}(x)$ is defined in \eqref{eq:sGi}.
The MT family of search directions for NSDP\,\eqref{al:nsdp} is defined formally as follows:
\begin{definition}
  The Monteiro-Tsuchiya (MT) family of search directions for NSDP\,\eqref{al:nsdp} is the family comprising directions $\Delta w=(\Delta x,\Delta Y, \Delta z)\in \mathcal{W}$ such that $(\Delta x,\Delta Y)$ solves the linear equation\,\eqref{eq:mtlinear} parameterized by a nonsingular scaling matrix $P$ and a barrier parameter $\mu$. 
\end{definition}
The MT family reduces to several important classes of directions by selecting the scaling matrix $P$ appropriately.
Below, we give some members of the MT family together with the corresponding scaling matrices $P$ and relevant equations on $\hat{G}(x)$ and $\hat{Y}$.  
\begin{description}
\item[{\bf MT direction}:] $P=I$,
$
\hat{G}(x)=G(x),\
\hat{Y} = Y.
$
\item[{\bf NT direction}:] $P=Y^{\frac{1}{2}}$,
$
\hat{G}(x)=Y^{\frac{1}{2}}G(x)Y^{\frac{1}{2}},\
\hat{Y} = I.
$
\item[{\bf {\HKM} direction}:] $P = G(x)^{-\frac{1}{2}}$,
$
  \hat{G}(x) = I,\ \hat{Y} = G(x)^{\frac{1}{2}}YG(x)^{\frac{1}{2}}.
$
\item[{\bf {\HKM}-dual direction}:]
$P = (YG(x)Y)^{\frac{1}{2}}$, 
$
  \hat{Y}\hat{G}(x)\hat{Y} = I.
$ 
\item[{\bf MTW direction}:] $P=W^{-\fr}$,
$
\hat{G}(x)=\hat{Y},
$
where $W$ is defined in \eqref{eq:W}.
\end{description}

The above NT and {\HKM} directions correspond with those derived by Yamashita's group \cite{yamashita2012local,yamashita2012primal} from
\eqref{eq:mz} by setting $P=W^{-\fr}, G(x)^{-\fr}$, respectively. The NT, {\HKM}, and {\HKM}-dual directions are also members of the MZ family.
Clear theoretical advantages and disadvantages of each direction when compared with others have not been elucidated in the context of the NSDP. 
But, some remarks on numerical aspects will be described in Section\,\ref{subsec:remarks}. 
\subsection{Motivation of using the MT family}
Let $P\in \R^{m\times m}$ be a scaling matrix and define $\Xi_{\mu}^P:\mathcal{W}\to \mathcal{W}$ by
\begin{equation}
\Xi_{\mu}^P(w):=\begin{bmatrix}
\nabla_xL(w)\\
\hG(x)^{\fr}\hY\hG(x)^{\fr}-\mu I\\
h(x) 
\end{bmatrix},
\label{eq:Xip}
\end{equation}
where $\mu>0$ is a barrier parameter and the function $L:\mathcal{W}\to \mathcal{W}$ is the Lagrangian function for NSDP\,\eqref{al:nsdp}.
Notice that if $\Xi_{\mu}^P(w)=0$ and $w\in \W_+$, $w$ is a BKKT triplet for $\mu>0$ and a KKT triplet for $\mu=0$. 
Denote the Jacobian of $\Xi_{\mu}^P(w)$ by $\mathcal{J}\Xi_{\mu}^P(w)$.
If the Newton equation $\mathcal{J}\Xi_{\mu}^P(w)\Delta w= - \Xi_{\mu}^P(w)$ is uniquely solvable,  
$\Delta w$ is a descent direction of $\|\Xi_{\mu}^P(w)\|^2=\|\nabla_xL(w)\|_2^2+\|\hG(x)^{\fr}\hY\hG(x)^{\fr}-\mu I\|_{\rm F}^2+\|h(x)\|_2^2$.
Hence, $\|\Xi_{\mu}^P(w)\|^2$ can be decreased by proceeding along $\Delta w$.
This implies that $\|\Xi_{\mu}^P(w)\|^2$ with $P=I$, namely $\|\Xi_{\mu}^I(w)\|^2$, can be decreased too, because $\|\Xi_{\mu}^P(w)\|^2=\|\Xi_{\mu}^I(w)\|^2$ follows from the fact that $\|G(x)^{\fr}YG(x)^{\fr}-\mu I\|_{\rm F}$ is scale invariant in the sense of 
$$
\|\hG(x)^{\fr}\hY\hG(x)^{\fr}-\mu I\|_{\rm F}=\|G(x)^{\fr}YG(x)^{\fr}-\mu I\|_{\rm F}.
$$
This property is significant since it implies huge potential for gaining a globally convergent PDIPM for the whole MT family by utilizing $\|\Xi_{\mu}^I(w)\|$ as a merit function to measure the deviation of $w$ from the set of BKKT triplets with barrier parameter $\mu$.
On the other hand, such a merit function for the MZ family has not yet been discovered.
In the existing works on PDIPMs using the MZ family\,\cite{yamakawa1,yamashita2012primal},
the global convergence to BKKT triplets is established for only a particular class of search directions produced with $P$ such that $\hG(x)$ and $\hY$ commute.
From this viewpoint, it is beneficial to develop the MT family in the context of the NSDP.   
}

\section{Primal-dual interior point method based on the MT family of search directions}\label{sec:3}
In this section, we present a PDIPM using the MT family of search directions.
Similar to many classical PDIPMs, the algorithm seeks to find a KKT triplet of NSDP~\eqref{al:nsdp} by closely tracking a so-called central path that is formed by BKKT triplets of NSDP~\eqref{al:nsdp} and driving a barrier parameter to zero. 
\subsection{Description of the proposed algorithm}
Recall $\Wcal_{++}=\{w\in \Wcal\mid G(x)\in \mS^m_{++},Y\in \mS^m_{++}\}$.
Given $\olmu>0$ and $\ol{w}\in \Wcal_{++}$, we produce two search directions $\Deltap \ol{w}$ and $\Deltac \ol{w}$ by solving certain scaled Newton equations based on the MT family which will be shown shortly.
We refer to steps along these directions as {\it tangential} and {\it centering} steps, respectively.
We then produce an intermediate point, say $\ol{w}_{\fr}$, and the next point, say $\ol{w}_+$,
by
\begin{align}
\ol{w}_{\fr}:=\ol{w} + s_{\rm t}\Deltap\ol{w},\ \ \ol{w}_+:=\ol{w}_{\fr} + s_{\rm c}\Deltac\ol{w},\label{al:wfrwplus}
\end{align}
where $s_{\rm t},s_{\rm c}\in (0,1]$ are step-sizes. The above-mentioned Newton equations are described as follows.
For $\mu>0$ and $P\in \R^{m\times m}$ being nonsingular,
let $\Xi_{\mu}^P:\W\to \W$ be the function defined in \eqref{eq:Xip} and consider the following linear equations:
\begin{subequations}
\begin{align}
&\mathcal{J}\Xi_{{\mu}}^P(w)\Deltap w =
\begin{bmatrix}
0\\
-{\mu} I\\
0
\end{bmatrix},\ 
\label{eq:pre3}\\
&\mathcal{J}\Xi_{{\mu}}^P(w)\Delta_{\rm c}w = -\Xi_{\mu}^P({w}),\label{eq:corr4}
\end{align}
\end{subequations}
where $\mathcal{J}\Xi_{{\mu}}^P(w)$ denotes the Jacobian of $\Xi_{{\mu}}^P(w)$ with respect to $w$. 
The scaled Newton equations that we solve
for obtaining the directions $\Deltap \ol{w}$ and $\Deltac \ol{w}$ are actually the linear equation\,\eqref{eq:pre3} with $(w,\mu)=(\ol{w},\ol{\mu})$
and equation\,\eqref{eq:corr4} with $(w,\mu)=(\ol{w}_{\fr},\ol{\mu}_+)$, respectively. Here, we define $$\olmu_+:=(1-s_t)\olmu.$$

Let us consider how $\ol{w}_{\fr}$ and $\olmu_+$ may be interpreted. 
Perturb the BKKT equation $\Xi_{{\mu}}^P(w)=0$ to 
$\Xi_{{\mu}}^P(w)=\Xi_{\olmu}^P(\ol{w})$ and assume $\mathcal{J}\Xi_{\olmu}^P(\ol{w})$ to be nonsingular. 
By the implicit function theorem, 
there exist a scalar $\delta_{\ol{\mu}}>0$
and a smooth curve 
$v(\cdot):[\olmu-\delta_{\ol{\mu}},\olmu+\delta_{\ol{\mu}}]\to \Wcal_{++}$
such that $v(\olmu)=\ol{w}$ and 
\begin{equation}
\Xi_{\mua}^P(v(\mu))=\Xi_{\olmu}^P(\ol{w}),\ \ \forall\mua\in [\olmu-\delta_{\ol{\mu}},\olmu+\delta_{\ol{\mu}}].
\label{eq:0104}
\end{equation}
Compared to $v(\olmu)$, 
$v(\olmu-\delta_{\ol{\mu}})$ can be a better approximation to a KKT triplet in view of the fact that $\|\Xi_{0}^P(v(\olmu-\delta_{\ol{\mu}}))\|$ is upper-bounded by a smaller quantity than that of $\|\Xi_{0}^P(v(\olmu))\|$.
Indeed,
\begin{align*}
\|\Xi_{0}^P(v(\olmu-\delta_{\ol{\mu}}))\|
&\le \|
\Xi_{\olmu-\delta_{\ol{\mu}}}^P(v(\olmu-\delta_{\ol{\mu}}))
\|+\|(\olmu-\delta_{\ol{\mu}})I\|_{\rm F}\\
&=\|\Xi_{\olmu}^P(\ol{w})\| + \sqrt{m} (\olmu-\delta_{\ol{\mu}})
\ \ \ (\mbox{by \eqref{eq:0104} with $\mua=\olmu-\delta_{\ol{\mu}}$}),
\end{align*}
while $\|\Xi_{0}^P(v(\olmu))\|=\|\Xi_{0}^P(\ol{w})\|\le \|\Xi_{\olmu}^P(\ol{w})\| + \sqrt{m}\olmu$.
Inspired by this observation, we trace the curve $v$ by making
a step along the tangential direction $\dot{v}(\olmu):=\frac{d}{d\mu}v(\olmu)$.
Actually, $\dot{v}(\olmu)$ is the solution of the equation
$\mathcal{J}\Xi_{\olmu}^P(v(\olmu))^{\ast}\dot{v}(\olmu)=[0,I,0]^{\top}$, which is obtained by differentiating equation~\eqref{eq:0104} with respect to $\mua$ at $\mua=\olmu$ and using the relation $\frac{d}{d\mu}{\Xi}_{\mu}^P(w)=[0,-I,0]^{\top}$.
Comparing this equation to \eqref{eq:pre3} with 
$(w,\mu):=(\ol{w},\ol{\mu})=(v(\ol{\mu}),\ol{\mu})$, we see $\Deltap \ol{w} = -\olmu\dot{v}(\olmu)$, and thus, \eqref{al:wfrwplus} and $v(\olmu)=\ol{w}$ imply
$$\ol{w}_{\fr}=v(\olmu)-s_{\rm t}\olmu\dot{v}(\olmu)\approx v((1-s_{\rm t})\olmu)=v(\olmu_+).$$
Accordingly, $\ol{w}_{\fr}$ may be interpreted as a first-order approximation of $v(\olmu_+)$.
The idea of the tangential direction is also discussed in \cite[Section~5]{forsgren2002interior} about PDIPMs for nonlinear optimization.

Next, let us consider what role the centering step plays.
Recall that $w\in \mathcal{W}_{++}$ which satisfies $\Xi_{\olmu_+}^P(w)=0$ is a BKKT triplet with $\olmu_+$.
In view of the fact that $\Deltac \ol{w}$ is a solution of the linear equation\,\eqref{eq:corr4} with $(w,\mu)=(\ol{w}_{\fr},\ol{\mu}_+)$, we may expect to get closer to the set of BKKT triplets with $\mua=\olmu_+$ by proceeding along $\Deltac \ol{w}$.

For the sake of analysis in the subsequent sections, 
we explicitly define equations~\eqref{eq:pre3} and \eqref{eq:corr4} below.
Here, $\Deltap w$ and $\Deltac w$ in \eqref{eq:pre3} and \eqref{eq:corr4}
are simply expressed as $\Delta w=(\Delta x,\Delta Y,\Delta z)$.
We have:
\begin{subequations}
\begin{align}
&\nabla^2_{xx}L(w)\Delta x  - \mathcal{J}G(x)^{\ast}\Delta Y
+\nabla h(x)\Delta z
=
\begin{cases}
0\ &\mbox{for \eqref{eq:pre3}}\\
- \nabla_xL(w)\ &\mbox{for \eqref{eq:corr4}},
\end{cases}
\label{al:scalednewton-1}
\\
&\hat{G}(x)^{\frac{1}{2}}\left(P^{-\top}\Delta {Y}P^{-1}\right)\hat{G}(x)^{\frac{1}{2}}
+ \hU\hat{Y}\hat{G}(x)^{\frac{1}{2}}
+ \hat{G}(x)^{\frac{1}{2}}\hat{Y}\hU
=
\begin{cases}
-{\mu} I\ &\mbox{for \eqref{eq:pre3}}\\
{\mu} I - \hat{G}(x)^{\frac{1}{2}}\hat{Y}\hat{G}(x)^{\frac{1}{2}}\ &\mbox{for \eqref{eq:corr4}},
\end{cases}
\label{al:scalednewton-2}\\
&\nabla h(x)^{\top}\Delta x = \begin{cases}
0\ \ &\mbox{for \eqref{eq:pre3}}\\
-h(x)\ \ &\mbox{for \eqref{eq:corr4}},
\end{cases}
\label{al:scalednewton-4}
\end{align}
\end{subequations}
where $\hU\in \mS^m$ is a solution of \eqref{al:1214-1} and represented as $\hU=\sum_{i=1}^n\Delta x_i\mathcal{L}_{\hG(x)^{\fr}}^{-1}(\hGi(x))$.

Our PDIPM is formally stated as in Algorithm\,\cref{alg4}, in which
scaling matrices $P$ are set through the function $\mathcal{P}:\W_{++}\to \R^{m\times m}$ such that $\mathcal{P}(w)$ is nonsingular for any $w\in \W_{++}$.
With appropriate $\mathcal{P}$,
we can set $P$ to either of MT, {\HKM}, {\HKM-dual}, {NT}, and MTW scaling matrices. For example,
we obtain the {\HKM}-dual scaling matrix $P$ by letting $\mathcal{P}(w):=(YG(x)Y)^{\fr}$.
In Section~\ref{subsec:mainr}, we give specific conditions on $ \mathcal{P}$ under which the local convergence can be established.
The step-sizes $s_{\rm t}$ and $s_{\rm c}$ in \eqref{al:wfrwplus} are computed as in Lines~\ref{al:skt} and \ref{al:skc} such that $G(x)\in \mS^m_{++}$ and $Y\in \mS^m_{++}$, i.e., updated points remain strictly feasible. 
Hereafter, we will study the local convergence of this algorithm. 
Under mild assumptions, we will prove that 
$\bar{\ell}_k=\bar{m}_{\kfr}=0$ eventually holds, which together with other several properties finally establishes the superlinear convergence to a KKT triplet. 

If the function $G$ is affine (e.g., $G(x)=F_0+\sum_{i=1}^nx_iF_i$ with $F_i\in \mS^m$ for $i=0,1,\ldots,n$),  
the step sizes $s_k^{\rm t}$ and $s_k^{\rm c}$ can be computed in a more systematic way.
Provided with a pair of search directions $(dx,dY)\in \R^n\times \mS^m$ and $(x,Y)$ such that $G(x),Y\in \mS^m_{++}$, 
we have $G(x+sdx), Y+sdY\in \mS^m_{++}$ for any $s\in [0,\bar{s}]$ with
\begin{equation}
\bar{s} = \min(1,s_x,s_Y),
\label{eq:sxsY}
\end{equation}
where
\begin{align*}
s_x&:=
\begin{cases}
\displaystyle{- \frac{0.99}{\lambda_{\rm min}(G(x)^{-1}\sum_{i=1}^ndx_iF_i)}} &\mbox{if }\lambda_{\rm min}(G(x)^{-1}\sum_{i=1}^ndx_iF_i)<0\\
1                                                     &\mbox{otherwise}, 
\end{cases}\\	
s_Y&:=
\begin{cases}
\displaystyle{- \frac{0.99}{\lambda_{\rm min}(Y^{-1}dY)}} &\mbox{if }\lambda_{\rm min}(Y^{-1}dY)<0\\
1                                                     &\mbox{otherwise.} 
\end{cases}
\end{align*}
From this fact, in Line~\ref{al:skt},
it suffices to find the smallest $\bar{\ell}_{k}$ such that 
$\beta^{\bar{\ell}_k}(1- \min(0.95,\mu_k)^{\alpha})\le \bar{s}$
after computing $\bar{s}$ at $w^k$.
In Line~\ref{al:skc}, $\bar{m}_{\kfr}$ can be computed in a similar fashion.

The current algorithm does not set any stopping criterion.
An instance of criterion is to stop the algorithm if $\|\Xi^I_0(w^k)\|\le \varepsilon$ holds, where $\varepsilon>0$ is a preset small parameter.

\paragraph{Globalization of Algorithm\,\cref{alg4}}
Under stringent assumptions, Algorithm\,\cref{alg4} may not be globally convergent as it is.
One common way for endowing it with the global convergence property is to combine Algorithm\,\cref{alg4} and globalization techniques based on a decent method using an appropriate merit function for NSDP~\eqref{al:nsdp}. 
For example, 
the algorithm can attain the global convergence property by inserting the following procedure just before Line~\ref{linescaling} in Algorithm\,\cref{alg4}:
\begin{quote} 
``If $\|\Xi_{\mu_k}^I(w^k)\|>r_k$, then find $v^{k}\in \W_{++}$ such that $\|\Xi_{\mu_k}^I(v^{k})\|\le r_k$ and replace $w^k$ by $v^k$'',
\end{quote}
where $\{r_k\}$ is a positive sequence converging to zero, where each term can be either adaptively determined at each $k$ or a prefixed value. 	
In a similar manner to \cite[Theorem~1]{yamashita2012primal}, 
we can show that $\{v^k\}$ accumulates at KKT triplets under some assumptions including the MFCQ and the boundedness of $x$-components of $\{v^k\}$.
For computing $v^k$, one way is to make use of the line search with $\|\Xi_{\mu_k}^I(w)\|^2$ as a merit function, which will be implemented in the section of numerical experiments.

\begin{algorithm}[tbh]
\caption{MT based primal-dual interior point method (Local algorithm)}
\label{alg4}
\begin{algorithmic}[1]
\Require
Choose
$\alpha,\beta\in (0,1)$ and 
$w^0=(x^0,Y_0,z^0)\in \mathcal{W}_{++}$, $\mu_0>0$,
and $\mathcal{P}:\W_{++}\to \R^{m\times m}$ such that $\mathcal{P}(w)$ is nonsingluar for any $w\in \W_{++}$.
Set $k \leftarrow0$.
\For{$k =0,1,2,\ldots,$}
\State{\bf Scaling:}\label{linescaling} 
Set $P_k\leftarrow \P(w^k)$
and scale $G(x^k)$ and $Y_k$ by 
$$
\widehat{G}(x^k)\leftarrow P_kG(x^k)P_k^{\top},\ \widehat{Y}_k\leftarrow P_k^{-\top}Y_kP_k^{-1}.
$$
\State{\bf Tangential Step:} 
Solve the scaled Newton equation\,\eqref{eq:pre3}
with $(w,\mu,P)= (w^k,\mu_k,P_k)$ to obtain
$\Deltap w^{k}=(\Deltap x^{k},\Deltap Y_{k},\Deltap z^{k})$.
\State\label{line:inn1} Find the smallest integer $\bar{\ell}_k\ge 0$ satisfying 
\begin{align*}
  &G\left(x^k+\beta^{\bar{\ell}_k}(1-\min\left(0.95,\mu_k\right)^{\alpha})\Deltap x^k\right)\in \mS^m_{++}, \notag \\
  &Y_k+\beta^{\bar{\ell}_k}(1-\min(0.95,\mu_k)^{\alpha})\Deltap Y_k\in \mS^m_{++}.
\end{align*}
\State\label{al:skt}
Set $s_k^{\rm t} \leftarrow \beta^{\bar{\ell}_k}(1- \min(0.95,\mu_k)^{\alpha})$ and $w^{k+\fr}\leftarrow w^k + s_k^{\rm t}\Deltap w^k$.
\State{\bf Update the barrier parameter:}
$\mu_{k+1} \leftarrow \left(1-s_k^{\rm t}\right)\mu_k$
\State{\bf Scaling:} 
Set $P_{k+\fr}\leftarrow \P(w^{k+\fr})$
and scale $G(x^{k+\fr})$ and $Y_{k+\fr}$ by 
$$
\widehat{G}(x^{k+\fr})\leftarrow P_{k+\fr}G(x^{k+\fr})P_{k+\fr}^{\top},\
\widehat{Y}_{k+\fr}\leftarrow P_{k+\fr}^{-\top}Y_{k+\fr}P_{k+\fr}^{-1}.
$$
\State{\bf Centering Step:} 
{Solve the scaled Newton equation\,\eqref{eq:corr4} with
  $(w,\mu,P)= (w^{k+\frac{1}{2}},\mu_{k+1},P_{k+\fr})$ to obtain $\Delta_{\rm c}w^{k+\fr}$.}
\State\label{line:inn2} Find the smallest integer ${\bar{m}_{\kfr}}\ge 0$ satisfying 
\begin{align*}
  &G(x^{k}+\beta^{{\bar{m}_{\kfr}}}\Delta_{\rm c}x^{\kfr})\in \mS^m_{++}, Y_{k}+\beta^{{\bar{m}_{\kfr}}}\Delta_{\rm c}Y_{\kfr}\in \mS^m_{++}.
\end{align*}
  \State\label{al:skc} {Set $s^{\rm c}_{\kfr}\leftarrow \beta^{{\bar{m}_{\kfr}}}$ and $w^{k+1}\leftarrow w^{\kfr}+s^{\rm c}_{\kfr}\Deltac w^{\kfr}$.}
  \State{\bf Update:}\ Set $k \leftarrow k+1$
\EndFor
\end{algorithmic}
\end{algorithm}

{\subsection{Some remarks on solving Newton equations}\label{subsec:remarks}
We remark that the system of equations\,\eqref{al:scalednewton-1}-\eqref{al:scalednewton-4} is equivalent to the following:
\begin{subequations}
\begin{align}
&\begin{bmatrix}
\nabla^2_{xx}L(w)+\mathcal{B}(x,Y)&\nabla h(x)\\
\nabla h(x)^{\top}& 0
\end{bmatrix}\begin{bmatrix}
\Delta x\\
\Delta z
\end{bmatrix}\notag\\
=&\begin{cases}
[\mu\mathcal{J}G(x)^{\ast}G(x)^{-1},0]^{\top}\\
[-\mu\mathcal{J}G(x)^{\ast}G(x)^{-1}-\nabla_xL(w),
-h(x)
]^{\top},
\end{cases} \label{al:20230416-1}   \\ 
&\Delta Y=-2\sum_{i=1}^n\Delta x_iP^{\top}{\rm Sym}\left(\hG(x)^{-\fr}
\mathcal{L}_{\hG(x)^{\fr}}^{-1}(\widehat{\G}_i(x))
\hY\right)P,\label{al:20230416-2}
\end{align}
\end{subequations}
where for each $i,j=1,2,\ldots,n$,
$$
\mathcal{B}(x,Y)_{i,j}
:=2{\rm Tr}\left(\widehat{\G}_i(x)\hG(x)^{-\fr}\mathcal{L}_{\hG(x)^{\fr}}^{-1}(\widehat{\G}_j(x))\hY\right).
$$
Hence, the system is uniquely solvable if and only if 
the coefficient matrix in the first equation is nonsingular.  In the numerical experiments, we will actually solve the above system. 	  
In building this system, we need to compute  
$\mathcal{B}(x,Y)$ whose $(i,j)$-th element for each direction is represented as follows, where the argument $x$ is omitted from $G(x)$ and $\G_i(x)$ for simplicity:
\begin{eqnarray*}
&{\rm MT}:
2{\rm Tr}\left( 
Y\G_iG^{-\fr}\mathcal{L}_{G^{\fr}}^{-1}(\G_j)
\right),
&\\
&{\rm NT}:\ {\rm Tr}(\widehat{\G}_i\hG^{-\fr}\widehat{\G}_j),\ 
\mbox{\HKM-dual}:
2{\rm Tr}(\widehat{G}^{-\fr}\widehat{\G}_i\widehat{G}^{-\fr}\mathcal{L}_{\hG^{\fr}}^{-1}(\widehat{\G}_j)),&\\
&{\rm \HKM}:
{\rm Tr}(\widehat{\G}_i\hY\widehat{\G}_j),\
{\rm MTW}:
2{\rm Tr}(\hG\widehat{\G}_i\hG^{-\fr}\mathcal{L}_{\hG^{\fr}}^{-1}(\widehat{\G}_j)).&
\end{eqnarray*}
These matrices $\mathcal{B}(x,Y)$ are symmetric except for the case of MT. 
In general, $\mathcal{B}(x,Y)\in \mS^m$ holds for any search direction corresponding to a scaling matrix $P$ such that $\hG(x)$ and $\hY$ commute.
Therefore, for example, when $f$ is strongly convex,
$G$ is affine, the constraint $h(x)=0$ is absent,
and such $P$ is employed, the coefficient matrix in \eqref{al:20230416-1} is symmetric positive definite, and hence an efficient linear equation algorithm such as the conjugate gradient method is applicable to equation\,\eqref{al:20230416-1}.
Meanwhile, computation of $\mathcal{L}_{(\cdot)}^{-1}(\cdot)$, which is contained in the MT, {\HKM}-dual, and MTW, is often costly.
The reader may be referred to \cite{monteiro1999implementation} for efficient implementation of the MT direction in the framework of the LSDP.
} 

\section{Local convergence analysis of Algorithm\,\cref{alg4}}\label{sec:local}
In this section, we prove local superlinear convergence of Algorithm\,\cref{alg4}.
The analysis is specific to the MT family of search directions.
In particular, its notable features are summarized as in the following remarks.
\begin{remark}\label{rem:1}
\begin{enumerate}
\item \label{rem1} The function $\Xi^P_{\mu}$ is not normally defined 
at a point satisfying $G(x)\in \mS^m\setminus \mS^m_{+}$ due to the presence of $G(x)^{\fr}$, and as a result, neither is its Jacobian on a boundary point, i.e., a point such that $G(x)\in \mS^m_+\setminus \mS^m_{++}$.
This fact causes difficulties on analyzing the limiting behavior of MT search directions.
\item \label{rem2} The analysis will be made under general assumptions (Assumption\,\cref{assum:A}) on scaling matrices used for producing MT search directions.
Whereas Yamashita and Yabe\,\cite{yamashita2012local} presented local convergence analysis for each of AHO, NT, and {\HKM} directions individually,
we can show the superlinear convergence result for a wider class of scaling matrices in a unified manner by virtue of these assumptions.
The assumptions are actually fulfilled by major MT members (Proposition\,\cref{prop:scaling}).
\end{enumerate}
\end{remark}
Henceforth, we assume that
the functions $f$, $G$, and $h$ are three times continuously differentiable, and let $w^{\ast}:=(x^{\ast},Y_{\ast},z^{\ast})$ be a KKT triplet of NSDP\,\eqref{al:nsdp} that satisfies the following regularity conditions. These conditions are all adopted in the recent researches\,\cite{yamashita2012local,yamakawa2,okuno2018primal} on PDIPMs for NSDPs:\vspace{1em}\\
\noindent{\bf Nondegeneracy condition:}
Let 
$r_{\ast}:={\rm rank}\,G(x^{\ast})$ and let $e_i\ (i=1,2,\ldots,m-r_{\ast})$ be an orthonormal basis of the null space of $G(x^{\ast})$. 
Then, the vectors 
$v_{ij}\in \R^n\ (1\le i\le j\le m-r_{\ast})$ and $\nabla h_i(x^{\ast})\ (i=1,2,\ldots,s)$ are linearly independent, where
$$
v_{ij}:=(e_i^{\top}\G_1(x^{\ast})e_j,\cdots, e_i^{\top}\G_n(x^{\ast})e_j)^{\top}\in \R^n\ (1\le i\le j\le m-r_{\ast}).
$$
\noindent{\bf Second-order sufficient condition}:
$\nabla_{xx}^2L(w^{\ast})+H(x^{\ast},Y_{\ast})$ is positive definite on the critical cone $C(x^{\ast})$ at $x^{\ast}$, which is defined by 
\begin{equation*}
C(x^{\ast}):= \left\{d\in \R^n\mid \nabla f(x^{\ast})^{\top}d=0, \nabla h(x^{\ast})^{\top}d=0,
\mathcal{J}G(x^{\ast})d\in T_{\mS^m_+}(G(x^{\ast}))
\right\}, \label{eq:critical}
\end{equation*}
where $T_{\mS^m_+}(G(x^{\ast}))$ denotes the tangent cone of $\mS^m_+$ at $G(x^{\ast})$ and 
$H(x^{\ast},Y_{\ast})$ is the matrix in $\mS^n$ whose entries are given by 
$$
(H(x^{\ast},Y_{\ast}))_{i,j}:= 2Y_{\ast}\bullet \Gi(x^{\ast})G(x^{\ast})^{\dag}\G_j(x^{\ast})
$$
for $i,j=1,2,\ldots,n$.
Here, $G(x^{\ast})^{\dag}$ denotes the Moore-Penrose inverse matrix of $G(x^{\ast})$.\\
\noindent{\bf Strict complementarity condition}: It holds that $G(x^{\ast})+Y_{\ast}\in \mS^m_{++}$.\vspace{0.5em}\\
For detailed explanations of the above three conditions, we refer readers to, e.g., \cite{yamashita2012local, shapiro1997first,bonnans2013perturbation}.
In what follows, we study the local convergence behavior of
Algorithm\,\cref{alg4} in a sufficiently small neighborhood of $w^{\ast}$.
For the sake of analysis, we introduce the following functions 
$\Phi^1_{\mu},\Phi^2_{\mu}:\mathcal{W}\to \mathcal{W}$:
\begin{equation}
{\Phi}^1_{\mu}(w):=\begin{bmatrix}
\nabla_xL(w)\\
G(x)Y-\mu I\\
h(x)
\end{bmatrix},\
{\Phi}^{2}_{\mu}(w):=\begin{bmatrix}
\nabla_xL(w)\\
{\rm Sym}\left(G(x)Y\right)-\mu I\\
h(x)
\end{bmatrix}.
\label{eq:defPhi}
\end{equation}
We often use the following relation: For any $\mu>0$ and $w\in \mathcal{W}$,  
\begin{equation}
\|\Phi_{\mu}^1(w)\|\ge \|\Phi_{\mu}^2(w)\|.\label{eq:0218}
\end{equation}
Under the above three regularity conditions of $w^{\ast}$, the Jacobian of ${\Phi}^2_{0}$ is nonsingular at $w^{\ast}$ by \cite[Corollary~1]{yamashita2012local}. Then, by the implicit function theorem along with the strict complementarity condition, we can ensure that there exist 
a scalar $\bar{\mu}>0$ and a continuous path $v^{\ast}:[0,\bar{\mu}]\to \mathcal{W}_{+}$ such that $v^{\ast}(0)=w^{\ast}$, it is smooth in $(0,\bar{\mu})$,
and $\Phi^2_{\mu}(v^{\ast}(\mu))=0\ (\mu\in [0,\bar{\mu}])$, which indicates that $v^{\ast}(\mu)=(x^{\ast}(\mu),Y_{\ast}(\mu),z^{\ast}(\mu))$ is a BKKT triplet with a barrier parameter $\mu\in (0,\bar{\mu}]$.
Hereafter, we often refer to the smooth path $v^{\ast}$ as the {\it central path} that emanates from $w^{\ast}$. 

Associated with the central path, we define the following set for $r>0$ and $\mu \in (0,\bar{\mu}]$ to measure the deviation of given iterates from the central path:
$$
\mathcal{N}^{r}_{\mu}:=\left\{w=(x,Y,z)\in \mathcal{W}_{++}\mid 
\|
\Phi^1_{\mu}(w)
\|
\le r
\right\}.
$$
The following proposition concerns an error bound on $\|\tw^{\ell}-w^{\ast}\|$ that will play a crucial role in the convergence analysis.
It follows immediately from \cite[Lemma~1]{yamashita2012local}.
\begin{proposition}\label{prop:0219-1}
Let $\{\tmu_{\ell}\}\subseteq \R_{++}$
and $\{r_{\ell}\}\subseteq \R_{++}$
 be a sequence such that $\lim_{\ell\to\infty}\tmu_{\ell}=0$ and $r_{\ell}=o(\tmu_{\ell})$.
Moreover, let $\{\tw^{\ell}\}\subseteq \mathcal{W}_{++}$ be a sequence satisfying $\tw^{\ell}\in \N^{r_{\ell}}_{\mu_{\ell}}$ for each $\ell$ and $\lim_{\ell\to\infty}\tw^{\ell}=w^{\ast}$.
Then,
$\|\tw^{\ell}-w^{\ast}\|=\Theta(\tmu_{\ell}).$
\end{proposition}
\subsection{Main results}\label{subsec:mainr}
To make our goal clear, we first present our main results without proofs.
We shall make the following assumptions:
\begin{assumption}\label{assum:A}
In Algorithm\,\cref{alg4}, the function $\P:\mathcal{W}_{++}\to \R^{m\times m}$ satisfies the following conditions:
\begin{enumerate}
\item[{\rm (\bf P1)}] $\P(w)$ is nonsingular for any $w\in \mathcal{W}_{++}$.
\item[{\rm (\bf P2)}]
Let $w^{\ast}$ be the KKT triplet defined in the beginning of this Section\,\cref{sec:local} and let $\tau>0$.
There exists $\xi\in (0,1)$ such that, for arbitrary sequences
$\{\tw^{\ell}\}:=\{(\tx^{\ell},\tY_{\ell},\tz^{\ell})\}\subseteq \mathcal{W}_{++}$,
$\{\tmu_{\ell}\}\subseteq \R_{++}$, and $\{r_{\ell}\}\subseteq\R_{++}$
satisfying
\begin{equation}
\lim_{\ell\to\infty}(\tw^{\ell},\tmu_{\ell})=(w^{\ast},0),\ \tw^{\ell}\in \N_{\tmu_{\ell}}^{r_{\ell}},\ r_{\ell}=\tau\tmu_{\ell}^{1+\xi},
\label{eq:P2}
\end{equation}
the sequence 
\begin{align*}
\{\zli\}:=\left\{\tmu_{\ell}{\P}(\tw^{\ell})^{-1}\mathcal{L}^{-1}_{\hat{G}_{\ell}^{\fr}}(\widehat{\G}_i(\tx^{\ell}))\hat{G}_{\ell}^{-\fr}{\P}(\tw^{\ell})\right\}\subseteq \R^{m\times m}
\end{align*}
is bounded for each $i=1,2,\ldots,n$, where
$$
G_{\ell}:=G(\tx^{\ell}),\ \hat{G}_{\ell}:=\P(\tilde{w}^{\ell})G_{\ell}\P(\tilde{w}^{\ell})^{\top},\ \widehat{\G}_i(\tx^{\ell}):=\P(\tilde{w}^{\ell})\Gi(\tx^{\ell})\P(\tilde{w}^{\ell})^{\top}.$$
\end{enumerate}
\end{assumption}
One may ask how general the above conditions~({\rm \bf P1}) and ({\rm \bf P2}) are.
We answer this question in the following proposition.
We defer its proof to Appendix.  
In the proof, we often make use of Proposition\,\cref{prop:xyx}, which will be presented in Section\,\cref{subsec:prel}.
\begin{proposition}\label{prop:scaling}
If the function $\P$ attains at any $w\in \mathcal{W}_{++}$
either {\rm (i)} the identity matrix or {\rm (ii)} the scaling matrix for {\HKM} direction,
it satisfies conditions~{\rm({\bf P1})} and {\rm({\bf P2})} for any $\xi\in (0,1)$.
Moreover, it also does so for any $\xi\in [\frac{1}{2},1)$ when
it attains at any $w$ the scaling matrix corresponding to any one of {\rm (iii)} NT, {\rm (iv)}
{\HKM}-dual, and {\rm (v)} MTW directions. 
\end{proposition}
The following theorem states the superlinear convergence of Algorithm\,\cref{alg4} started from a point that stays in a small neighborhood of the central path and is furthermore sufficiently close to the KKT triplet $w^{\ast}$.
The proof will be given in Section\,\cref{sec:last} after preparing some preliminary results in the subsequent section.
In the theorem, 
$\alpha$ is the constant chosen in the initial setting of Algorithm\,\cref{alg4}. 
In addition, let $\xi\in (0,1)$ and $\tau>0$ be the constants described in condition~{({\bf P2})} and also $\xi^{\prime}\in (0,1)$ be a constant satisfying
\begin{align}
\frac{\xi-\alpha}{1+\alpha}>\xi^{\prime}>\frac{\xi}{2},\ 0<\alpha<\frac{\xi}{\xi+2}.
\label{al:cond}
\end{align}
For example, the above conditions are fulfilled by $(\xi,\xi^{\prime},
\alpha)=(1/2,1/3,1/10)$.
\begin{theorem}\label{thm:main}
Suppose that Assumption\,\cref{assum:A} holds and $w^0$ is sufficiently close to $w^{\ast}$ and satisfies $w^0\in \N^{\tau\mu_0^{1+\xi}}_{\mu_0}$ with $\mu_0 = \gamma \|\Xi_0^I(w^0)\|$ for fixed $\gamma>0$. 
Then,
\begin{enumerate}
\item\label{mThm:1}
it holds that, for each $k\ge 0$, the linear equation~\eqref{eq:pre3} with
$(w,\mu,P)=(w^{k},\mu_k,P_k)$
and equation\,\eqref{eq:corr4} with
$(w,\mu,P)=(w^{k+\fr},\mu_{k+1},P_{k+\fr})$ are uniquely solvable and 
\begin{subequations}
\begin{align}
&\mu_{k+1}=\mu_k^{1+\alpha}<\mu_k,&\label{eq:0704-1721}\\
&s_k^{\rm t}=1-\mu_k^{\alpha},\ w^{k+\fr}=w^k+(1-\mu_k^{\alpha})\Deltap w^k\in
\N_{\mu_{k+1}}^{\tau\mu_{k+1}^{1+\xi^{\prime}}},\label{eq:0702-2329}\\
&s_{\kfr}^{\rm c}=1,\ w^{k+1} = w^{k+\fr}+\Deltac w^{\kfr}\in
\N_{\mu_{k+1}}^{\tau\mu_{k+1}^{1+{\xi}}}.\label{eq:0704-1726}
\end{align}
\end{subequations}
\item\label{mThm:2}Furthermore, the generated sequence $\{w^k\}$ converges to $w^{\ast}$ superlinearly at the order of $1+\alpha\in (1,4/3)$, namely,
$$
\|w^{k+1}-w^{\ast}\|=\rO(\|w^k-w^{\ast}\|^{1+\alpha}).
$$
\end{enumerate}
\end{theorem}
\subsection{Preliminary results for Theorem\,\cref{thm:main}}\label{subsec:prel}
In this section, we consider arbitrary sequences 
$\{\tw^{\ell}\}\subseteq \Wcal_{++}$, $\{\tmu_{\ell}\}\subseteq \R_{++}$, and $\{r_{\ell}\}\subseteq\R_{++}$
satisfying \eqref{eq:P2} in condition~({\bf \rm P2}).
In addition, we define a sequence of nonsingular matrices $\{\tP_{\ell}\}\subseteq \R^{m\times m}$ by $\tP_{\ell}:=\P(\tw^{\ell})$ for each $\ell\ge 0$, where $\P:\mathcal{W}_{++}\to \R^{m\times m}$ is a matrix-valued function satisfying conditions~({\bf P1}) and ({\bf P2}).
We will discuss properties of these sequences.
Note that the above sequences are 
not necessarily ones generated by Algorithm\,\cref{alg4}.
They are just introduced for the sake of analysis of the algorithm.

\subsubsection{Uniqueness and error bounds of the solutions of the linear equations~\eqref{eq:pre3} and \eqref{eq:corr4}}\label{sec:limit}
In this section, we aim to show the following proposition concerning 
unique solvability of the linear equations~\eqref{eq:pre3} and \eqref{eq:corr4} near $w^{\ast}$
and error bounds of their solutions.
\begin{proposition}\label{prop:0613-2}
For any $\ell\ge \ell_0$ with $\ell_0$ sufficiently large, the linear equations~\eqref{eq:pre3} and \eqref{eq:corr4}
with $(P,w,\mu)=(\tP_{\ell},\tw^{\ell},\tmu_{\ell})$ have unique solutions, say
$\Deltap \tw^{\ell}$ and $\Deltac \tw^{\ell}$, respectively.
In particular, we have
\begin{enumerate}
\item\label{e1} $\|\Deltap \tw^{\ell}\|=\rO(\tmu_{\ell})$, and
\item\label{e2} $\|\Deltac \tw^{\ell}\|=\rO(\tmu_{\ell}^{1+\xi})$.
\end{enumerate}
\end{proposition}
Error bounds and uniqueness of solutions of the Newton equations play key roles in the local convergence analysis of the PDIPMs using the MZ family\cite{yamashita2012local,yamakawa2,okuno2018primal}, in which the proofs are based on the nonsingularity of the Jacobian of the function $\Phi_{0}^2$. 
However, as explained in Remark\,\cref{rem1}, the Jacobian of 
$\Xi_0^P$ is unavailable at the KKT triplet $w^{\ast}$ in general and thus a different approach is necessary for proving Proposition\,\cref{prop:0613-2}.

To prove Proposition\,\cref{prop:0613-2}, we define the following linear mapping $T_{w,P}:\mathcal{W}\to \mathcal{W}$
for a triplet $w=(x,Y,z)\in \mathcal{W}$ and a nonsingular matrix $P\in \R^{m\times m}$: 
$$
T_{w,P}(\Delta w):=
\begin{bmatrix}
&\nabla^2_{xx}L(w)\Delta x- \mathcal{J}{G}(x)^{\ast}\Delta Y+\nabla h(x)\Delta z\\
&
{\rm Sym}\left(G(x)\Delta Y+
\sum_{i=1}^n\Delta x_i
\mathcal{S}_{P}^i(x,Y)
\right)\\
&\nabla h(x)^{\top}\Delta x
\end{bmatrix},
$$
where $\Delta w=(\Delta x,\Delta Y,\Delta z)\in \mathcal{W}$ and  
\begin{align}
\mathcal{S}_{P}^i(x,Y)&:=
P^{-1}\left(\hG(x)^{\fr}\hU_i\hat{Y}
+\hG(x)\hat{Y}\hU_i\hG(x)^{-\fr}\right)P,\notag \\
\hU_i&:=\mathcal{L}_{\hG(x)^{\fr}}^{-1}(\hGi(x))\label{al:ui2}
\end{align}
for $i=1,2,\ldots,n$.
We then consider the following linear equations:
\begin{subequations}
\begin{align}
T_{w,P}(\Deltap w) &= \begin{bmatrix}
0\\
-\mu I\\
0
\end{bmatrix}\label{al:0608-1},\\
T_{w,P}(\Deltac w) &= \begin{bmatrix}
-\nabla_xL(w)\\
-{\rm Sym}{\left(G(x)Y\right)}+{\mu} I\\
-h(x)
\end{bmatrix}. \label{al:0608-2}
\end{align}
\end{subequations}
\begin{lemma}\label{lem:0612}
Solving equation\,\eqref{eq:pre3} is equivalent to solving equation\,\eqref{al:0608-1}, and the same holds for solving equations\,\eqref{eq:corr4} and \eqref{al:0608-2}.
\end{lemma}
\begin{proof}
We show this lemma only for equations\,\eqref{eq:corr4} and \eqref{al:0608-2} because the equivalency of \eqref{eq:pre3} and \eqref{al:0608-1} can be shown in a similar manner.
To this end, it suffices to show that
the solution set of the second component equation of \eqref{eq:corr4} (see also the second line of \eqref{al:scalednewton-2})
is identical to that of \eqref{al:0608-2}, namely, we show 
\begin{align}
&{\rm Sym}\left(G(x)\Delta Y+
\sum_{i=1}^n\Delta x_i
\mathcal{S}_{P}^i(x,Y)
\right)=\mu I - {\rm Sym}{\left(G(x)Y\right)} \label{al:0608-3}\\
&\Longleftrightarrow
\hG(x)^{\fr}\Delta \hY \hG(x)^{\fr} + 
2{\rm Sym}\left(
\sum_{i=1}^n\Delta x_i\hU_i\hY\hG(x)^{\fr}
\right)
=\mu I-\hG(x)^{\fr}\hY\hG(x)^{\fr},
\label{eq:0422-2}
\end{align}
where $\Delta \hY:=P^{-\top}\Delta Y P^{-1}$.
We first show the direction $(\Leftarrow)$. 
Multiplying $P^{-1}\hG(x)^{\fr}$ and $\hG(x)^{-\fr}P$ from the left and right sides on both sides of \eqref{eq:0422-2}, respectively, we have 
\begin{align*}
P^{-1}\hG(x)\Delta \hY P+
\sum_{i=1}^n
\Delta x_i\mathcal{S}_P^i(x,Y)=\mu I - P^{-1}\hG(x)\hY P,\notag
\end{align*}
which can be rewritten as $G(x)\Delta Y + \sum_{i=1}^n\Delta x_i\mathcal{S}_{P}^i(x,Y)=\mu I - G(x)Y$
via \eqref{eq:scalegxy} and $\Delta \hY=P^{-\top}\Delta Y P^{-1}$.
Symmetrizing this equation readily implies equation\,\eqref{al:0608-3}.

We next show the converse direction $(\Rightarrow)$.
For each $i=1,2,\ldots,n$, 
let 
$$Q_i:=
2P^{\top}
{\rm Sym}\left(
\hat{G}(x)^{-\fr}\hU_i\hat{Y}
\right)P\in \mS^m.$$ 
By taking into account \eqref{eq:scalegxy} and the fact that $G(x)P^{\top}\hG(x)^{-\fr}=P^{-1}\hG(x)^{\fr}$, it follows that, 
$$
\mathcal{S}_{P}^i(x,Y)=
G(x)Q_i.
$$
Then, in terms of $\mathcal{L}_{G(x)}$, equation\,\eqref{al:0608-3} is equivalently transformed as
\begin{equation}
\mathcal{L}_{G(x)}\left(\Delta Y+\sum_{i=1}^n\Delta x_iQ_i+Y\right)=2\mu I.\label{eq:0608-4}
\end{equation}
Note the fact that given $A\in \mS^m_{++}$ and $B\in \mS^m$, the linear matrix equation $\mathcal{L}_AX=B$ has a unique solution in $X\in \mS^m$.
Then, equation \eqref{eq:0608-4} together with $G(x)\in \mS^m_{++}$ 
and $\mathcal{L}_{G(x)}^{-1}I=G(x)^{-1}/2$ implies
$\Delta Y = \mu G(x)^{-1}-\sum_{i=1}^n\Delta x_iQ_i-Y$. Multiplying $\hG(x)^{\fr}P^{-\top}$ and $P^{-1}\hG(x)^{\fr}$ from the left and right sides of both sides of this equation, respectively,
and recalling \eqref{eq:scalegxy} again yield the desired equation\,\eqref{eq:0422-2}.
This completes the proof.
\hfill$\Box$
\end{proof}
The above lemma indicates 
that, to examine the linear equations \eqref{eq:pre3} and \eqref{eq:corr4}, 
we may consider \eqref{al:0608-1} and \eqref{al:0608-2} alternatively.
This is beneficial because, as stated in the following proposition, the limit operator of the sequence $\{T_{\tw^{\ell},\tP_{\ell}}\}_{\ell\ge 0}$
coincides with the Jacobian of $\Phi^2_{0}$ at $w^{\ast}$ (recall \eqref{eq:defPhi}), which is actually a one-to-one and onto mapping. 
\begin{proposition}\label{prop:20190605-1}
It holds that {\rm (1)}
$$
\lim_{\ell\to \infty}T_{\tw^{\ell},\tP_{\ell}} = \mathcal{T}_{\ast},
$$
where $\mathcal{T}_{\ast}:\mathcal{W}\to \mathcal{W}$ is the linear mapping defined by 
$$
\mathcal{T}_{\ast}(\Delta w):=\begin{bmatrix}
\nabla^2_{xx} L(w^{\ast})\Delta x-\mathcal{J}G(x^{\ast})^{\ast}\Delta Y+\nabla h(x^{\ast})\Delta z \\
{\rm Sym}\left(G(x^{\ast})\Delta Y + \sum_{i=1}^n\Delta x_i\Gi(x^{\ast})Y_{\ast}\right)\\
\nabla h(x^{\ast})^{\top}\Delta x
\end{bmatrix}
$$
for any $\Delta w\in \mathcal{W}$.
{\rm (2)} Furthermore, $\mathcal{T}_{\ast}$ is a one-to-one and onto mapping. 
\end{proposition}
\begin{proof}
{\rm (1)} We show only the convergence on the second component of $T_{\tw^{\ell},\tP_{\ell}}$, namely,  
for any $\Delta w=(\Delta x,\Delta Y,\Delta z)$, we prove  
\begin{align}
\lim_{\ell\to \infty}
{\rm Sym}\left(G_{\ell}\Delta Y+
\sum_{i=1}^n\Delta x_i
\mathcal{S}_{P}^i(\tilde{x}^{\ell},\tY_{\ell})\right)={\rm Sym}\left(G(x^{\ast})\Delta Y+
\sum_{i=1}^n\Delta x_i\Gi(x^{\ast})Y_{\ast}\right).\notag
\end{align}
To prove this equation, since $\lim_{\ell\to\infty}\Gl=G(x^{\ast})$, it suffices to show that 
\begin{align}
\lim_{\ell\to\infty}\mathcal{S}^i_{\tP_{\ell}}(\tilde{x}^{\ell},\tY_{\ell})=\Gi(x^{\ast})Y_{\ast}\label{eq:2342}
\end{align}
for an arbitrary index $i\in \{1,2,\ldots,n\}$.
Choose $i\in \{1,2,\ldots,n\}$ arbitrarily.
Note that, for any $w\in \mathcal{W}$, we have 
\begin{align*}
\mathcal{S}_{P}^i(x,Y)&=
P^{-1}\left(
\hGi(x)\hY
-
\hU_i\hG(x)^{\fr}\hY
+
\hG(x)\hY\hU_i\hG(x)^{-\fr}
\right)P \\
=&
P^{-1}\left(
\hGi(x)\hY-\hU_i\hG(x)^{\fr}\hY+\mu\hU_i\hG(x)^{-\fr}+(\hG(x)\hY-\mu I)\hU_i\hG(x)^{-\fr}
\right)P\\
=&
P^{-1}\left(
-\hU_i\hG(x)^{-\fr}(\hG(x)\hY-\mu I)+
\hGi(x)\hY+(\hG(x)\hY-\mu I)\hU_i\hG(x)^{-\fr}
\right)P\\
=&-P^{-1}\hU_i\hG(x)^{-\fr}P(G(x)Y-\mu I)+\Gi(x) Y+(G(x)Y-\mu I)P^{-1}\hU_i\hG(x)^{-\fr}P,
\end{align*}
where the first equality follows because  $\hG(x)^{\fr}\hU_i=\hGi(x)-\hU_i\hG(x)^{\fr}$ by \eqref{al:ui2}.
Since $\|\Gl\tY_{\ell}-\tmu_{\ell} I\|_F\le \tau \tmu_{\ell}^{1+\xi}$ from $\twl\in N_{\tmu_{\ell}}^{\tau \tmu_{\ell}^{1+\xi}}$, it holds that
\begin{align*}
\Phi_{i,\ell}
&:=\left\|-\tP_{\ell}^{-1}\mathcal{L}_{\hGl^{\fr}}^{-1}\left(\hGi(\txl)\right)\hGl^{-\fr}\tP_{\ell}\left(\Gl\tY_{\ell}-\tmu_{\ell} I\right)+\left(\Gl\tY_{\ell}-\tmu_{\ell} I\right)\tP_{\ell}^{-1}
\mathcal{L}_{\hGl^{\fr}}^{-1}\left(\hGi(\txl)\right)
\hGl^{-\fr}\tP_{\ell}\right\|_F\\
&\le
2\left\|\tP_{\ell}^{-1}
\mathcal{L}_{\hGl^{\fr}}^{-1}\left(\hGi(\txl)\right)
\hGl^{-\fr}\tP_{\ell}\right\|_F\left\|\Gl\tY_{\ell}-\tmu_{\ell} I\right\|_F\\
&\le 2\tmu_{\ell}\left\|\tP_{\ell}^{-1}
\mathcal{L}_{\hGl^{\fr}}^{-1}\left(\hGi(\txl)\right)
\hGl^{-\fr}\tP_{\ell}\right\|_F\cdot \frac{\tau \tmu_{\ell}^{1+\xi}}{\tmu_{\ell}}\\
&=\rO(\tmu_{\ell}^{\xi}),
\end{align*}
where the second equality holds from the boundedness of $\{\zli\}$ assumed in condition~({\bf P2}).
Hence, noting $\xi>0$ and $\lim_{\ell\to\infty}\tmu_{\ell}=0$, we have $\lim_{\ell\to
\infty}\Phi_{i,\ell} = 0$.
This fact readily yields $\|\mathcal{S}^i_P(\txl,\tYl)-\Gi(x^{\ast})Y_{\ast}\|_F\le \Phi_{i,\ell}+\|\Gi(\txl)\tYl-\Gi(x^{\ast})Y_{\ast}\|_F\to 0$ as $\ell\to \infty$. This means that \eqref{eq:2342} is valid. 
Since $i$ was arbitrary, we obtain the desired consequence.

{\rm (2)}
Notice that $\mathcal{T}_{\ast}$ is nothing but the Jacobian of
the function $\Phi^2_0(w)$ at the KKT triplet $w^{\ast}$.
Based on this fact, we can prove that $\mathcal{T}_{\ast}$ is a one-to-one and onto mapping in a manner similar to \cite[Theorem~1]{yamashita2012local} in the presence of the nondegeneracy, second-order sufficient, and strict complementarity conditions at $w^{\ast}$, which were assumed in the beginning of the section.
\end{proof}

Now, by means of Proposition\,\cref{prop:20190605-1},
let us prove Proposition\,\cref{prop:0613-2}.
\subsubsection*{Proof of Proposition\,\cref{prop:0613-2}}\label{sec:proof_thm}
By Proposition\,\cref{prop:20190605-1}(2), 
by taking $\ell$ sufficiently large, 
$\mathcal{T}_{\tw^{\ell},\tP_{\ell}}$ is a one-to-one and onto mapping and hence equations \eqref{al:0608-1} and \eqref{al:0608-2} have unique solutions. Therefore, by Lemma\,\cref{lem:0612},
equation~\eqref{eq:pre3} with $(P,w,\mu)=(\tP_{\ell},\tw^{\ell},\tmu_{\ell})$ turn out to posses the following solution\,\eqref{al:0415-22020} at $\tw^{\ell}$ uniquely, and the same relation holds between \eqref{eq:corr4} and \eqref{al:0415-12020}:
\begin{subequations}
\begin{align}
\Deltap \tw^{\ell}&=T_{\tw^{\ell},\tP_{\ell}}(\tw^{\ell})^{-1}\eta^{\rm t}_{\ell},\label{al:0415-22020}\\
\Deltac \tw^{\ell}&=T_{\tw^{\ell},\tP_{\ell}}(\tw^{\ell})^{-1}\eta^{\rm c}_{\ell},\label{al:0415-12020}
\end{align}
\end{subequations}
where 
$$
\eta^{\rm t}_{\ell}:=\begin{bmatrix}
0\\
- \tmu_{\ell}I \\
0
\end{bmatrix},\ 
\eta^{\rm c}_{\ell}:=\begin{bmatrix}
-\nabla_{x}L(\tw^{\ell})\\
\tmu_{\ell}I - {\rm Sym}(G_{\ell}\tY_{\ell}) \\
 - h(\tilde{x}^{\ell})
\end{bmatrix}
$$
for each $\ell$.
We next show the second-half claim.
It follows that 
\begin{equation}
\|\eta^{\rm c}_{\ell}\|=\rO(\tmu_{\ell}^{1+\xi})\label{eq:0614-1}
\end{equation}
from
$\twl\in \N_{\tmul}^{\tau\tmul^{1+\xi}}$ and 
$\|\tmu_{\ell}I - {\rm Sym}(G_{\ell}\tY_{\ell})\|_F\le
\|\tmu_{\ell}I -G_{\ell}\tY_{\ell}\|_F$.
Moreover, we have 
\begin{equation}
\|\eta^{\rm t}_{\ell}\|=\sqrt{m}\tmu_{\ell}\label{eq:0614-2}
\end{equation}
by definition.
As the limit of $\{T_{\tw^{\ell},\tP_{\ell}}\}$ is a one-to-one and onto mapping by Proposition\,\cref{prop:20190605-1}(2) again,
equations \eqref{al:0415-22020} and 
\eqref{al:0415-12020} derive
$\|\Deltap \tw^{\ell}\|=\Theta(\|\eta^{\rm t}_{\ell}\|)$ and $\|\Deltac \tw^{\ell}\|=\Theta(\|\eta^{\rm c}_{\ell}\|)$, which together with \eqref{eq:0614-1} and \eqref{eq:0614-2}
imply
\begin{equation*}
\|\Deltap \tw^{\ell}\|=\rO(\tmu_{\ell}),\ 
\|\Deltac \tw^{\ell}\|=\rO(\tmu_{\ell}^{1+\xi}).
\end{equation*} 
The proof is complete.
\hfill $\square$

\subsubsection{Effectiveness of tangential and centering steps}
In this section, we give crucial properties holding at the next points we move to in the tangential and centering steps.
Specifically, we show the following two propositions, where 
$\alpha,\xi$, and $\xi^{\prime}$ are constants satisfying \eqref{al:cond}.
Henceforth,
$\ell_0$ is the positive integer defined in Proposition\,\cref{prop:0613-2} and moreover,
$\Deltap \tw^{\ell}$ and $\Deltac \tw^{\ell}$ are the unique solutions of the linear equations\,\eqref{eq:pre3} and \eqref{eq:corr4} with $(P,w,\mu)=(\tP_{\ell},\tw^{\ell},\tmu_{\ell})$, respectively, for each $\ell\ge \ell_0$.

The first proposition claims that
the next point after performing the tangential step is eventually accommodated by
$\N_{\tmu_{\ell+\fr}}^{\tau\tmu_{\ell+\fr}^{1+\xi^{\prime}}}$, which contains $\N_{\tmu_{\ell+\fr}}^{\tau\tmu_{\ell+\fr}^{1+\xi}}$ because $\xi>\frac{\xi-\alpha}{1+\alpha}>\xi^{\prime}$.
\begin{proposition}\label{prop:0614}
The following properties hold:
\begin{enumerate}
\item\label{prop:0614-2}   
Choose a sequence $\{s_{\ell}\}\subseteq (0,1]$ arbitrarily. 
For $\ell\ge \ell_0$, we have
$
\|G(\tilde{x}^{\ell}+s_{\ell}\Deltap \tilde{x}^{\ell})(\tY_{\ell}+s_{\ell}\Deltap\tY_{\ell})-\mu(s_{\ell})I\|_F=\rO(\tmu_{\ell}^{1+\xi})
$
with $\mu(s_{\ell}):=(1-s_{\ell})\tmu_{\ell}$.
\item\label{prop:0614-1}
It holds that 
\begin{equation}
G(\tilde{x}^{\ell}+(1-\tmu_{\ell}^{\alpha})\Deltap \tilde{x}^{\ell})\in \mS^m_{++},\ \tY_{\ell}+(1-\tmu_{\ell}^{\alpha})\Deltap\tY_{\ell}\in \mS^m_{++} \label{eq:1517}
\end{equation}
for any sufficiently large $\ell\ge \ell_0$.
\item\label{prop:0614-3} Let $\twt^{\ell+\fr}:=\tw^{\ell} + (1-\tmu_{\ell}^{\alpha})\Deltap \tw^{\ell}$ and
$\tmu_{\ell+\fr}:=\tmu_{\ell}^{1+\alpha}$ for each $\ell$.
Then,
$\twt^{\ell+\fr}\in \N_{\tmu_{\ell+\fr}}^{\tau\tmu_{\ell+\fr}^{1+\xi^{\prime}}}$
holds for any sufficiently large $\ell\ge \ell_0$.
\end{enumerate}
\end{proposition}
The second proposition claims that 
the next point after performing the centering step is eventually accommodated by
$\N_{\tmu_{\ell}}^{\tau\tmu_{\ell}^{1+2\kappa}}$, which is contained by  $\N_{\tmu_{\ell}}^{\tau\tmu_{\ell}^{1+\xi}}$ if $\frac{\xi}{2}<\kappa$.
\begin{proposition}\label{prop:0626}
The following properties hold:
\begin{enumerate}
\item\label{prop:0626-2}
Choose a sequence $\{s_{\ell}\}\subseteq (0,1]$ arbitrarily.
Then, we have
$$\|G(\tilde{x}^{\ell}+s_{\ell}\Deltac \tilde{x}^{\ell})(\tY_{\ell}+s_{\ell}\Deltac \tY_{\ell})-\tmu_{\ell}I\|_F=\rO((1-s_{\ell})\tmu_{\ell}^{1+\xi}+2s_{\ell}\tmu_{\ell}^{1+2\xi}).$$
\item\label{prop:0626-1}
For any sufficiently large $\ell\ge \ell_0$, we have
\begin{equation}
G(\tilde{x}^{\ell}+\Deltac \tilde{x}^{\ell})\in \mS^m_{++},\ \tY_{\ell}+\Deltac \tY_{\ell}\in \mS^m_{++}. \label{eq:2213}
\end{equation}
\item\label{prop:0626-3}
Choose $0<\kappa<\xi$ and let $\twc^{\ell+\fr}:=\tw^{\ell} +\Deltac \tw^{\ell}$ for each $\ell$. Then, $\twc^{\ell+\fr}\in 
\N_{\tmu_{\ell}}^{\tau\tmu_{\ell}^{1+2\kappa}}$
holds for any sufficiently large $\ell\ge \ell_0$.
\end{enumerate}
\end{proposition}
In order to prove the preceding two propositions, we require the following two results.
\begin{proposition}\label{prop:order}
Let $F:\R^p\to \R^q$ be a twice continuously differentiable function
and $\{(v^{\ell},\Delta v^{\ell})\}\subseteq \R^p\times \R^p$ be a sequence converging to some point $(v^{\ast},0)\in \R^p\times \R^p$. Then, 
$\left\|F(v^{\ell}+\Delta v^{\ell})-F(v^{\ell})-\nabla F(v^{\ell})^{\top}\Delta v^{\ell}
\right\|_2 = \rO(\|\Delta v^{\ell}\|_2^2)$.
\end{proposition}
\begin{proof}
By Taylor's theorem, 
$
F_i(v^{\ell}+\Delta v^{\ell})-F_i(v^{\ell})-\nabla F_i(v^{\ell})^{\top}\Delta v^{\ell}
=\int_0^1 (1-t)(\Delta v^{\ell})^{\top}\nabla^2F_i(v^{\ell}+t\Delta v^{\ell})\Delta v^{\ell}dt
$
for $i=1,2,\ldots,p$, which yields the assertion.
\hfill$\Box$
\end{proof}
\begin{proposition}\label{prop:xyx}
Let $X\in \mS^m_+$, $Y\in \mS^m$, and $\mu\in \R$.
Then, 
$\|X^{\fr}YX^{\fr}-\mu I\|^2_{\rm F} 
+\frac{\|XY-YX\|^2_{\rm F}}{2}
= \|XY-\mu I\|^2_{\rm F}$, and thus 
$\|X^{\fr}YX^{\fr}-\mu I\|_{\rm F}\le \|XY-\mu I\|_{\rm F}$ holds.
\end{proposition}
\begin{proof}
{For matrices $A,B\in \mS^m$ and a scalar $\mu\in \R$, it follows that
\begin{equation}
\|ABA-\mu I\|_{\rm F}^2+\frac{\|A^2B-BA^2\|^2_{\rm F}}{2}=\|A^2B-\mu I\|^2_{\rm F}, \label{eq:2023-1}
\end{equation}
from 
\begin{align}
\|A^2B-\mu I\|_{\rm F}^2-\|ABA-\mu I\|_{\rm F}^2
&={\rm Tr}(A^2B^2A^2-ABA^2BA)\notag\\
&=\frac{1}{2}{\rm Tr}(BA^4B+A^2B^2A^2-BA^2BA^2-A^2BA^2B)\notag\\
&=\frac{1}{2}{\rm Tr}\left((BA^2-A^2B)(A^2B-BA^2)\right)\notag\\
&=\frac{1}{2}\|A^2B-BA^2\|_{\rm F}^2,\notag
\end{align}
where the second equality follows from 
$
{\rm Tr}(BA^4B)={\rm Tr}(A^2B^2A^2)
$
and ${\rm Tr}(ABA^2BA)={\rm Tr}(BA^2BA^2)={\rm Tr}(A^2BA^2B)$. By setting $(A,B):=(X^{\fr},Y)$ 
with $(X,Y)\in \mS^m_{+}\times \mS^m$
in \eqref{eq:2023-1}, we obtain the desired assertion.} 
\hfill $\Box$
\end{proof}

We are now in a position to prove Propositions\,\cref{prop:0614} and \cref{prop:0626}.
\subsubsection*{Proof of Proposition\,\cref{prop:0614}}
We make frequent use of the following facts from Proposition\,\cref{prop:0613-2}\,\eqref{e1}:
\begin{equation}
\|\Deltap\txl\|_2=\rO(\tmul),\ \|\Deltap\tYl\|_{\rm F}=\rO(\tmul),\ \|\Deltap\twl\|=\rO(\tmul).
\label{eq:mousukoshi2}
\end{equation}
For each $\ell$, let $\Deltap x_i^{\ell}\in \R$ be the $i$-th element of 
$\Deltap x^{\ell}$, 
\begin{align}
\Deltap G_{\ell}:=\sum_{i=1}^n\Deltap x_i^{\ell}\Gi(\tilde{x}^{\ell}),\
\Deltap\hGl:=\sum_{i=1}^n\Deltap x_i^{\ell}\hGi(\tilde{x}^{\ell})=P_{\ell}\Deltap \Gl P_{\ell}^{\top}.
\label{al:ast3}
\end{align}
Define
$R_{\ell}(\Delta\tilde{x}):=G(\tilde{x}^{\ell}+\Delta x)-\Gl-\sum_{i=1}^n\Delta	 x_i\Gi(\tilde{x}^{\ell})\in S^m$
for any $\Delta x\in \R^n$. By \eqref{eq:mousukoshi2} and Proposition\,\cref{prop:order} with $F$ replaced by $G$,
\begin{equation}
\|R_{\ell}(s_{\ell}\Deltap \tilde{x}^{\ell})\|_{\rm F}=
\rO(s_{\ell}^2\tmu^2_{\ell}).
\label{eq:0701-1}
\end{equation}
Now, we first show item\,\eqref{prop:0614-2}. 
Letting $V_{\tP_{\ell}}^{\rm t}:=\tP_{\ell}^{-1}\mathcal{L}_{\hGl^{\fr}}^{-1}(\Deltap \hG_{\ell})\hG_{\ell}^{-\fr}\tP_{\ell}$
for each $\ell$ yields
\begin{align}
\|V_{\tP_{\ell}}^{\rm t}\|_{\rm F}             &\le \frac{\|\Deltap \tilde{x}^{\ell}\|_2}{\tmu_{\ell}} \sum_{i=1}^n\tmu_{\ell}\|\tP_{\ell}^{-1}
              \mathcal{L}_{\hGl^{\fr}}^{-1}(\hGi(\tilde{x}^{\ell}))\hGl^{-\fr}
              \tP_{\ell}\|_{\rm F}
               =\rO(1),\label{al:1}
\end{align}
where  
the last equality is due to $\|Z_{\ell}^i\|_{\rm F}=\rO(1)$ assumed in ({\bf P2}) and 
$\|\Deltap \txl\|=\rO(\tmu_{\ell})$ from \eqref{eq:mousukoshi2}.

By the twice continuous differentiability and Taylor's expansion of the function $G$, it holds that 
\begin{align}
&\|G(\tilde{x}^{\ell}+s_{\ell}\Deltap \tilde{x}^{\ell})(\tY_{\ell}+s_{\ell}\Deltap \tY_{\ell})-\mu(s_{\ell})I\|_{\rm F}\notag\\
=&\|(G_{\ell}+s_{\ell}\Deltap G_{\ell}+R_{\ell}(s_{\ell}\Deltap \tilde{x}^{\ell}))(\tY_{\ell}+s_{\ell}\Deltap \tY_{\ell})-\mu(s_{\ell})I\|_{\rm F}
\notag
\\
\le&\|G_{\ell}\tY_{\ell}+s_{\ell}
(\Deltap G_{\ell}\tY_{\ell}+G_{\ell}\Deltap\tYl)+s_{\ell}^2\Deltap G_{\ell}\Deltap \tY_{\ell}- \mu(s_{\ell})I\|_{\rm F} +\rO(s_{\ell}^2\tmu_{\ell}^2),\label{al:0614-1}
\end{align}
where we used \eqref{eq:0701-1} and $\|\tYl+s_{\ell}\Deltap\tYl\|_{\rm F}=\rO(1)$, which can be ensured by using $s_{\ell}\in (0,1]$, $\lim_{\ell\to \infty}\tY_{\ell} = Y_{\ast}$, and $\|\Deltap \tYl\|_{\rm F}=\rO(\tmul)$ from \eqref{eq:mousukoshi2}.
We have $\|G_{\ell}\tY_{\ell}-\tmu_{\ell} I\|_{\rm F}=\rO(\tmu_{\ell}^{1+\xi})$ from $\twl\in\N_{\tmul}^{\tau\tmul^{1+\xi}}$ and, by $\hU_{\ell}^{\rm t}:=\mathcal{L}_{\hGl^{\fr}}^{-1}(\Deltap \hGl)$ for each $\ell$, 
\begin{align}
&
\|G_{\ell}\tY_{\ell}+s_{\ell}
(\Deltap G_{\ell}\tY_{\ell}+G_{\ell}\Deltap\tYl)+s_{\ell}^2\Deltap G_{\ell}\Deltap \tY_{\ell}- \mu(s_{\ell})I\|_{\rm F}\notag\\ 
=&\|G_{\ell}\tY_{\ell}+s_{\ell}^2\Deltap G_{\ell}\Deltap \tY_{\ell}- \mu(s_{\ell}) I \notag\\
                   &\hspace{3em}+
\underbrace{s_{\ell}
\tP_{\ell}^{-1}(\hU_{\ell}^{\rm t}\hG_{\ell}^{\fr}+\hG_{\ell}^{\fr}\hU_{\ell}^{\rm t})\tP_{\ell}^{-\top}\tY_{\ell}}_{=s_{\ell}\Deltap G_{\ell}\tY_{\ell}\ \because 
\hU_{\ell}^{\rm t}=\mathcal{L}_{\hGl^{\fr}}^{-1}(\Deltap \hGl),\ \eqref{al:ast3}}
+\underbrace{
s_{\ell}(
-\tmu_{\ell} I-\tP_{\ell}^{-1}\hGl^{\fr}\hU_{\ell}^{\rm t} \tP_{\ell}^{-\top}\tY_{\ell}-
G_{\ell}\tY_{\ell}\tP_{\ell}^{-1}\hU_{\ell}^{\rm t}\hG_{\ell}^{-\fr}\tP_{\ell}
)}_{=s_{\ell}G_{\ell}\Deltap\tY_{\ell}\ \because\ \mbox{the first line of }\eqref{al:scalednewton-2} 
}
\|_{\rm F}\notag\\
=&\|
G_{\ell}\tY_{\ell}+s_{\ell}^2\Deltap G_{\ell}\Deltap \tY_{\ell}-\mu(s_{\ell}) I+s_{\ell}\left(
-\tmu_{\ell} I + \tP_{\ell}^{-1}\hU_{\ell}^{\rm t}\hG_{\ell}^{\fr}\tP_{\ell}^{-\top}\tY_{\ell}-G_{\ell}\tY_{\ell}\tP_{\ell}^{-1}\hU_{\ell}^{\rm t}\hG_{\ell}^{-\fr}\tP_{\ell}
\right)
\|_{\rm F}\notag \\
=& \left\|G_{\ell}\tY_{\ell}-\tmu_{\ell} I + s_{\ell}^2\Deltap G_{\ell}\Deltap \tY_{\ell}+s_{\ell}\left(V_{\tP_{\ell}}^{\rm t}(G_{\ell}\tY_{\ell}-\tmu_{\ell} I)-(G_{\ell}\tY_{\ell}-\tmu_{\ell} I)V_{\tP_{\ell}}^{\rm t}\right)\right\|_{\rm F}\notag \\
\le&\|G_{\ell}\tY_{\ell}-\tmu_{\ell} I\|_{\rm F} + s_{\ell}^2\|\Deltap G_{\ell}\|_{\rm F}\|\Deltap \tY_{\ell}\|_{\rm F}+ 2s_{\ell}\|V^{\rm t}_{\tP_{\ell}}\|\|G_{\ell}\tY_{\ell}-\tmu_{\ell} I\|_{\rm F}\notag \\
=&\rO(\tmu_{\ell}^{1+\xi}), \label{al:0614-2}
\end{align}
where the last equality follows from 
\eqref{al:1}, 
$\|G_{\ell}\tY_{\ell}-\tmu_{\ell} I\|_{\rm F}=\rO(\tmu_{\ell}^{1+\xi})$, $\{s_{\ell}\}\subseteq (0,1]$, and $0<\xi<1$ for each $\ell$.
By noting $\{s_{\ell}\}\subseteq (0,1]$ and $0<\xi<1$ again,
equations\,\eqref{al:0614-1} and \eqref{al:0614-2} yield
$$
\|G(\tilde{x}^{\ell}+s_{\ell}\Deltap \tilde{x}^{\ell})(\tY_{\ell}+s_{\ell}\Deltap \tY_{\ell})-\mu(s_{\ell})I\|_{\rm F}= \rO(\tmu_{\ell}^{1+\xi}).
$$

We next show item\,\eqref{prop:0614-1}.
To derive a contradiction, we assume to the contrary that there exists an {\it infinite} subsequence $\{\tw^{\ell}\}_{\ell\in \mathcal{L}}\subseteq \{\tw^{\ell}\}$
such that $G(\tilde{x}^{\ell}+(1-\tmu_{\ell}^{\alpha})\Deltap \tilde{x}^{\ell})\in S^m_{++}$ and $\tY_{\ell}+(1-\tmu_{\ell}^{\alpha})\Deltap \tY_{\ell}\in S^m_{++}$ do not hold for each $\ell\in \mathcal{L}$.
Then, by noting $\{(\Gl,\tY_{\ell})\}_{\ell\in \mathcal{L}}\subseteq S^m_{++}\times S^m_{++}$,  there exists some sequence $\{\hat{t}_{\ell}\}_{\ell\in \mathcal{L}}\subseteq (0,1]$ such that, for each $\ell\in \mathcal{L}$, 
\begin{equation}
G(\tilde{x}^{\ell}+\hat{t}_{\ell}(1-\tmu_{\ell}^{\alpha})\Deltap \tilde{x}^{\ell})\in S^m_{+},\ \tY_{\ell}+\hat{t}_{\ell}(1-\tmu_{\ell}^{\alpha})\Deltap \tY_{\ell}\in S^m_+
\label{eq:20200812-1}
\end{equation}
hold and, moreover, one of the above lies on the boundary of $S^m_+$, namely, 
\begin{equation}
G(\tilde{x}^{\ell}+\hat{t}_{\ell}(1-\tmu_{\ell}^{\alpha})\Deltap \tilde{x}^{\ell})\in S^m_{+}\setminus S^m_{++}\ \mbox{or }\tY_{\ell}+\hat{t}_{\ell}(1-\tmu_{\ell}^{\alpha})\Deltap \tY_{\ell}\in S^m_{+}\setminus S^m_{++}
\label{eq:20200812-2}
\end{equation}
holds.
Denote
\begin{align*}
A_{\ell,1}&:=G\left(\tilde{x}^{\ell}+\hat{t}_{\ell}(1-\tmu_{\ell}^{\alpha})\Deltap \tilde{x}^{\ell}\right)^{\fr}\left(\tY_{\ell}+\hat{t}_{\ell}(1-\tmu_{\ell}^{\alpha})\Deltap \tY_{\ell}\right)G\left(\tilde{x}^{\ell}+\hat{t}_{\ell}(1-\tmu_{\ell}^{\alpha})\Deltap \tilde{x}^{\ell}\right)^{\fr},\\
A_{\ell,2}&:=G\left(\tilde{x}^{\ell}+\hat{t}_{\ell}(1-\tmu_{\ell}^{\alpha})\Deltap \tilde{x}^{\ell}\right)\left(\tY_{\ell}+\hat{t}_{\ell}(1-\tmu_{\ell}^{\alpha})\Deltap \tY_{\ell}\right)
\end{align*}
for each $\ell$.
Expressions\,\eqref{eq:20200812-1} and \eqref{eq:20200812-2}
yield, for each $\ell\in \mathcal{L}$,
\begin{equation}
\lambda_{\min}\left(A_{\ell,1}\right)=0.\label{lamW}
\end{equation}
Proposition\,\cref{prop:xyx} and item\,\eqref{prop:0614-2}, which we have just shown above, with $\{s_{\ell}\}$ replaced by $\{\hat{t}_{\ell}(1-\tmu_{\ell}^{\alpha})\}_{\ell\in \mathcal{L}}$ imply
\begin{equation}
\|A_{\ell,1}-(1-\hat{t}_{\ell}(1-\tmu_{\ell}^{\alpha}))\tmu_{\ell}I\|_{\rm F}\le 
\|A_{\ell,2}-(1-\hat{t}_{\ell}(1-\tmu_{\ell}^{\alpha}))\tmu_{\ell}I\|_{\rm F}=\rO(\tmu_{\ell}^{1+\xi}).\label{eq:0628-3}
\end{equation}
On the other hand, by recalling that, for any $A\in S^m$, $\|A\|_{\rm F}^2$ is equal to the summation of squares of the eigenvalues of $A$, 
it holds that
$$
\|A_{\ell,1}-(1-\hat{t}_{\ell}(1-\tmu_{\ell}^{\alpha}))\tmu_{\ell}I\|_{\rm F}\ge\left|
\lambda_{\min}(A_{\ell,1})-(1-\hat{t}_{\ell}(1-\tmu_{\ell}^{\alpha}))\tmu_{\ell})\right|=
(1-\hat{t}_{\ell}(1-\tmu_{\ell}^{\alpha}))\tmu_{\ell},
$$
where the last equality follows from \eqref{lamW} and $0<\widehat{t}_{\ell}\le 1$.
Combining this inequality with \eqref{eq:0628-3} yields $(1-\widehat{t}_{\ell}(1-\tmu_{\ell}^{\alpha}))\tmu_{\ell}=\rO(\tmu_{\ell}^{1+\xi})$, which implies 
the boundedness of 
$\{|\tmu_{\ell}^{-\xi}-\hat{t}_{\ell}(\tmu_{\ell}^{-\xi}-\tmu_{\ell}^{\alpha-\xi})|\}_{\ell\in \mathcal{L}}$.
However,  since $\tmu_{\ell}\to 0$ as $\ell\in \mathcal{L}\to \infty$,
we have $0<\tmu_{\ell}^{\alpha}<1$ for $\ell$ sufficiently large, thus by $0<\alpha<\xi<1$ from \eqref{al:cond} and $\widehat{t}_{\ell}\le 1$, we obtain 
\begin{align}
|\tmu_{\ell}^{-\xi}-\hat{t}_{\ell}(\tmu_{\ell}^{-\xi}-\tmu_{\ell}^{\alpha-\xi})|&=
(1-\hat{t}_{\ell})\tmu_{\ell}^{-\xi} + \hat{t}_{\ell}\tmu_{\ell}^{\alpha-\xi}\notag \\
&\ge (1-\hat{t}_{\ell})\tmu_{\ell}^{\alpha-\xi} + \hat{t}_{\ell}\tmu_{\ell}^{\alpha-\xi}\notag \\
&=\tmu_{\ell}^{\alpha-\xi}\rightarrow \infty\ \left(\ell\left(\in \mathcal{L}\right)\to \infty\right),\notag 
\end{align}
which contradicts the boundedness of
$\{|\tmu_{\ell}^{-\xi}-\hat{t}_{\ell}(\tmu_{\ell}^{-\xi}-\tmu_{\ell}^{\alpha-\xi})|\}_{\ell\in \mathcal{L}}$.
Therefore, we conclude that \eqref{eq:1517} eventually holds.

Lastly, we prove item\,\eqref{prop:0614-3}.
Since $1+\xi^{\prime}<\frac{1+\xi}{1+\alpha}$ by condition\,\eqref{al:cond}, it holds that
\begin{equation}
\tmu_{\ell}^{1+\xi}=\ro(\tmu_{\ell+\fr}^{1+\xi^{\prime}}).\label{eq:miss}
\end{equation}
Let us first show
\begin{equation}
\|h(\txt^{\ell+\fr})\|_2=\ro(\tmu_{\ell+\fr}^{1+\xi^{\prime}}). \label{eq:2257-2}
\end{equation}
To this end, note that 
\begin{align}
\|h(\txt^{\ell+\fr})\|_2
=&\|h(\tilde{x}^{\ell})+(1-\tmu_{\ell}^{\alpha})\nabla h(\tilde{x}^{\ell})^{\top}\Deltap \tilde{x}^{\ell}\|_2 + \rO((1-\tmu_{\ell}^{\alpha})^2\|\Deltap\tilde{x}^{\ell}\|_2^2) \notag\\
=&\rO\left(\tmu_{\ell}^{1+\xi}+(1-\tmu_{\ell}^{\alpha})^2\tmu_{\ell}^{2}\right),\label{al:0629-1}
\end{align}
where
the first equality follows from Proposition\,\cref{prop:order} with $F$ replaced by $h$ and the second one owes to the facts that 
$\nabla h(\tilde{x}^{\ell})^{\top}\Deltap \tilde{x}^{\ell}=0$ (see the first line of \eqref{al:scalednewton-4}), $\|h(\tilde{x}^{\ell})\|_2=\rO(\tmul^{1+\xi})$ from $\twl\in\N_{\tmul}^{\tau\tmul^{1+\xi}}$, and $\|\Deltap \tilde{x}^{\ell}\|_2=\rO(\tmu_{\ell})$ from \eqref{eq:mousukoshi2}.
We finally conclude \eqref{eq:2257-2} from \eqref{al:0629-1}, $0<\xi^{\prime}<1$, and \eqref{eq:miss}.
Next, we will prove
\begin{equation}
\|\nabla_xL(\twt^{\ell+\fr})\|_2=\ro(\tmu_{\ell+\fr}^{1+\xi^{\prime}}).\label{eq:0629-2}
\end{equation}
Note that 
$\twt^{\ell+\fr}=\tw^{\ell}+(1-\tmu^{\alpha})\Deltap\tw^{\ell}$ and 
the Lagrangian $L$ is three times differentiable by assumption.
Noting $\Deltap \tw^{\ell}=\rO(\tmu_{\ell})$ from \eqref{eq:mousukoshi2} and using Proposition\,\cref{prop:order} with $F$ replaced by $\nabla_x L$, we obtain
\begin{align}
  \|\nabla_xL(\twt^{\ell+\fr})\|_2=&\left\|\nabla_xL(\tw^{\ell})+(1-\tmu_{\ell}^{\alpha})
\left(\nabla_{xx}^2L(\tw^{\ell})\Deltap \tilde{x}^{\ell}
-\mathcal{J}G(\tilde{x}^{\ell})^{\ast}\Deltap \tY_{\ell}+\nabla h(\tilde{x}^{\ell})^{\top}\Deltap \tz^{\ell}\right)\right\|_2 \notag\\
&\hspace{5em}+\rO((1-\tmu_{\ell}^{\alpha})^2\|\Deltap \tw^{\ell}\|^2)\notag\\
=&\|\nabla_xL(\tw^{\ell})\|_2+
\rO((1-\tmu_{\ell}^{\alpha})^2\|\Deltap \tw^{\ell}\|^2)\notag\\
=&\rO\left(\tmu_{\ell}^{1+\xi}+(1-\tmu_{\ell}^{\alpha})^2\tmu_{\ell}^2\right),\label{al:0629-3}
\end{align}
where the second equality follows from the first line of \eqref{al:scalednewton-1} with $\Delta w = \Deltap \tw^{\ell}$ and $w=\tw^{\ell}$ 
and the last equality is due to
$\|\nabla_xL(\tw^{\ell})\|_2=\rO(\tmul^{1+\xi})$ from $\tw^{\ell}\in \N_{\tmul}^{\tau\tmul^{1+\xi}}$.
By $0<\xi^{\prime}<1$ and \eqref{eq:miss}, \eqref{al:0629-3} implies the desired relation\,\eqref{eq:0629-2}.
By setting $s_{\ell}=1-\tmul^{\alpha}$ in item~\eqref{prop:0614-2} for any $l$ sufficiently large and using \eqref{eq:miss} again,
we have $\left\|G(\tx^{\ell}+(1-\tmul^{\alpha})\Deltap \txl)\left(\tYl+(1-\tmul^{\alpha})\Deltap\tYl\right)-\tmu_{\ell}^{1+\alpha}I\right\|_{\rm F}=\ro(\tmu_{\ell+\fr}^{1+\xi^{\prime}})$.
Combined with this fact and item~\eqref{prop:0614-1}, expressions \eqref{eq:2257-2} and \eqref{eq:0629-2} readily imply item~\eqref{prop:0614-3}. The proof is complete.
$\hfill \Box$

We next prove Proposition\,\cref{prop:0626}. The flow of the proof is actually similar to that of Proposition\,\cref{prop:0614}.
\subsubsection*{Proof of Proposition\,\cref{prop:0626}}
We make frequent use the following facts derived from
Proposition\,\cref{prop:0613-2}\,\eqref{e2}:
\begin{equation}
\|\Deltac \tilde{x}^{\ell}\|_2=\rO(\tmu_{\ell}^{1+\xi}),\ \|\Deltac \tY_{\ell}\|_{\rm F}=\rO(\tmu_{\ell}^{1+\xi}).
\label{eq:mousukoshi}
\end{equation}
For each $\ell$, let
$\Deltac \tx^{\ell}_i\in \R$ be the $i$-th element of $\Deltac \tx^{\ell}$,
$\Deltac\Gl:=\sum_{i=1}^n\Deltac \tx^{\ell}_i\Gi^{\ell}(\tx^{\ell})$, and 
$\Deltac\hGl:=\sum_{i=1}^n\Deltac \tx^{\ell}_i\hGi^{\ell}(\tx^{\ell})$.
We first show item~\eqref{prop:0626-2}.
By letting
$V^{\rm c}_{\tP_{\ell}}:=\tP_{\ell}^{-1}\mathcal{L}_{\hG_{\ell}^{\fr}}^{-1}(\Deltac \hG_{\ell})\hG_{\ell}^{-\fr}\tP_{\ell}$,
we have
\begin{align}
\|V^{\rm c}_{\tP_{\ell}}\|_{\rm F}             
              \le \frac{\|\Deltac \tilde{x}^{\ell}\|_2}{\tmu_{\ell}} \sum_{i=1}^n\tmu_{\ell}\|\tP^{-1}_{\ell}
\mathcal{L}_{\hG_{\ell}^{\fr}}^{-1}(\hGi^{\ell})\hG^{-\fr}_{\ell}\tP_{\ell}\|_{\rm F}
               =\rO(\tmu_{\ell}^{\xi}),\label{al:vcp}
\end{align}
where the last equality follows from \eqref{eq:mousukoshi} and the boundedness of $\{\zli\}$ assumed in ({\bf P2}).
Writing $\hU_{\ell}^{\rm c}:=\mathcal{L}_{\hGl^{\fr}}^{-1}(\Deltac \hGl)$ and 
choosing $\{s_{\ell}\}\subseteq (0,1]$ arbitrarily, we obtain 
\begin{align}
&\|
G_{\ell}\tY_{\ell}+s_{\ell}\Deltac G_{\ell}\tY_{\ell}+ s_{\ell}G_{\ell}\Deltac \tY_{\ell}-\tmu_{\ell} I
\|_{\rm F}\notag\\
&=\|
G_{\ell}\tY_{\ell}-\tmu_{\ell} I +\underbrace{s_{\ell}
\tP_{\ell}^{-1}\left(\hU_{\ell}^{\rm c}\hG_{\ell}^{\fr}+\hG_{\ell}^{\fr}\hU_{\ell}^{\rm c}\right)\tP_{\ell}^{-\top}\tY_{\ell}
}_{=s_{\ell}\Deltac{G}_{\ell}\tY_{\ell}\ \because\
\Deltac{G}_{\ell}=\tP_{\ell}^{-1}\Deltac{\hG}_{\ell}
\tP_{\ell}^{-\top},\ \hG_{\ell}^{\fr}\hU_{\ell}^{\rm c} + \hU_{\ell}^{\rm c}\hG_{\ell}^{\fr}=\Deltac\hG_{\ell}
}\notag\\
&\hspace{3em}+
\underbrace{s_{\ell}\left(
-G_{\ell}\tY_{\ell}+\tmu_{\ell} I-\tP_{\ell}^{-1}\hG_{\ell}^{\fr}\hU_{\ell}^{\rm c}\tP_{\ell}^{-\top}\tY_{\ell}-G_{\ell}\tY_{\ell}\tP_{\ell}^{-1}\hU_{\ell}^{\rm c}\hG_{\ell}^{-\fr}\tP_{\ell}\right)}_{=s_{\ell}G_{\ell}\Deltac\tY_{\ell}\ \because\ \mbox{the second line of }\eqref{al:scalednewton-2}}\|_{\rm F}\notag \\
&\le (1-s_{\ell})\|
G_{\ell}\tY_{\ell}-\tmu_{\ell} I\|_{\rm F}
+s_{\ell}\|\tP_{\ell}^{-1}\hU_{\ell}^{\rm c}\hG_{\ell}^{\fr}\tP_{\ell}^{-\top}\tY_{\ell}-G_{\ell}\tY_{\ell}\tP_{\ell}^{-1}\hU_{\ell}^{\rm c}\hG_{\ell}^{-\fr}\tP_{\ell}
\|_{\rm F}\notag\\
&=(1-s_{\ell})
\|
G_{\ell}\tY_{\ell}-\tmu_{\ell} I\|_{\rm F}
+s_{\ell}\|V_{\tP_{\ell}}^{\rm c}(G_{\ell}\tY_{\ell}-\tmu_{\ell} I)-(G_{\ell}\tY_{\ell}-\tmu_{\ell} I)V_{\tP_{\ell}}^{\rm c}\|_{\rm F}\notag \\
&=(1-s_{\ell})
\|
G_{\ell}\tY_{\ell}-\tmu_{\ell} I\|_{\rm F}
+2s_{\ell}\|V_{\tP_{\ell}}^{\rm c}\|_{\rm F}\|G_{\ell}\tY_{\ell}-\tmu_{\ell} I\|_{\rm F}\notag\\
&=\rO\left((1-s_{\ell})\tmu_{\ell}^{1+\xi}+2s_{\ell}\tmu_{\ell}^{1+2\xi}\right),\label{al:0629-4}
\end{align}
where the last equality follows from \eqref{al:vcp} and $\|\Gl\tYl-\tmul I\|_{\rm F}\le \tau\tmul^{1+\xi}$ by $\twl\in \N_{\tmul}^{\tau\tmul^{1+\xi}}$.        
Note \eqref{eq:mousukoshi}.
Then, in a manner similar to the proof for Proposition\,\cref{prop:0614}\eqref{prop:0614-2}, by applying Taylor's expansion to $G(\tilde{x}^{\ell}+s_{\ell}\Deltac \tilde{x}^{\ell})$ and also by using Proposition\,\cref{prop:order} and \eqref{al:0629-4}, we can derive 
\begin{align}
\|G(\tilde{x}^{\ell}+s_{\ell}\Deltac \tilde{x}^{\ell})(\tY_{\ell}+s_{\ell}\Deltac \tY_{\ell})-\tmu_{\ell} I\|_{\rm F}
=\rO\left((1-s_{\ell})\tmu_{\ell}^{1+\xi}+2s_{\ell}\tmu_{\ell}^{1+2\xi}\right).
\notag
\end{align}
Hence, we have item~\eqref{prop:0626-2}.

We next show item~\eqref{prop:0626-1}.
For contradiction, suppose to the contrary.
That is, we suppose that there exists an {\it infinite} subsequence $\{\tw^{\ell}\}_{\ell\in \mathcal{L}}\subseteq \{\tw^{\ell}\}$
such that $G(\tilde{x}^{\ell}+\Deltac \tilde{x}^{\ell})\in S^m_{++}$ and $\tY_{\ell}+\Deltac \tY_{\ell}\in S^m_{++}$ do not hold for each $\ell\in \mathcal{L}$.
Then, by noting $\{(G(\tilde{x}^{\ell}),\tY_{\ell})\}\subseteq S^m_{++}\times S^m_{++}$,  there exists some sequence $\{\hat{t}_{\ell}\}_{\ell\in \mathcal{L}}\subseteq (0,1]$ such that, for each $\ell\in \mathcal{L}$, 
\begin{equation*}
G(\tilde{x}^{\ell}+\hat{t}_{\ell}\Deltac \tilde{x}^{\ell})\in S^m_{+},\ \tY_{\ell}+\hat{t}_{\ell}\Deltac \tY_{\ell}\in S^m_+
\end{equation*}
hold and, moreover,  
\begin{equation}\label{eq:4.37}
G(\tilde{x}^{\ell}+\hat{t}_{\ell}\Deltac \tilde{x}^{\ell})\in S^m_{+}\setminus S^m_{++}\ \mbox{or }\tY_{\ell}+\hat{t}_{\ell}\Deltac \tY_{\ell}\in S^m_{+}\setminus S^m_{++}
\end{equation}
holds.
Denoting 
\begin{align*}
B_{\ell,1}&:=G(\tilde{x}^{\ell}+\hat{t}_{\ell}\Deltac \tilde{x}^{\ell})^{\fr}\left(\tY_{\ell}+\hat{t}_{\ell}\Deltac \tY_{\ell}\right)G(\tilde{x}^{\ell}+\hat{t}_{\ell}\Deltac \tilde{x}^{\ell})^{\fr},\\
B_{\ell,2}&:=G(\tilde{x}^{\ell}+\hat{t}_{\ell}\Deltac \tilde{x}^{\ell})\left(\tY_{\ell}+\hat{t}_{\ell}\Deltac \tY_{\ell}\right)
\end{align*}
for each $\ell$, we obtain from \eqref{eq:4.37} that
\begin{equation}
\lambda_{\min}\left(B_{\ell,1}\right)=0\label{lamW2}.
\end{equation}
Proposition\,\cref{prop:xyx} and the above item~\eqref{prop:0626-2} with $\{s_{\ell}\}$ replaced by $\{\tl\}_{\ell\in \mathcal{L}}$ imply
\begin{equation}
\|B_{\ell,1}-\tmu_{\ell}I\|_{\rm F}\le 
\|B_{\ell,2}-\tmu_{\ell}I\|_{\rm F}=\rO((1-\tl)\tmu_{\ell}^{1+\xi}+2\tl \tmu_{\ell}^{1+2\xi}).\label{eq:0629-3}
\end{equation}
On the other hand, the symmetry of $B_{\ell,1}$ and \eqref{lamW2} derive
$$
\|B_{\ell,1}-\tmu_{\ell}I\|_{\rm F}\ge \left|\lambda_{\min}(B_{\ell,1})-\tmu_{\ell}\right|=\tmu_{\ell}.
$$
Combining this inequality with \eqref{eq:0629-3} yields $\tmu_{\ell}=O((1-\tl)\tmu_{\ell}^{1+\xi}+2\tl \tmu_{\ell}^{1+2\xi})$, which further implies boundedness of 
$$
\left
\{\left|\frac{\tmu_{\ell}}{(1-\tl)\tmu_{\ell}^{1+\xi}+2\tl\tmu_{\ell}^{1+2\xi}}\right|
\right\}
=
\left
\{\left|\frac{1}{(1-\tl)\tmu_{\ell}^{\xi}+2\tl\tmu_{\ell}^{2\xi}}\right|
\right\}.
$$
However, by $0<\xi<1$ and
$\tmu_{\ell}\to 0$ as $\ell(\in \mathcal{L})\to \infty$, we see
\begin{equation*}
\left|\frac{1}{(1-\tl)\tmu_{\ell}^{\xi}+2t_{\ell}\tmu_{\ell}^{2\xi}}\right|
=\frac{1}{
\left|
\tmu_{\ell}^{\xi}+t_{\ell}(2\tmu_{\ell}^{2\xi}-\tmu_{\ell}^{\xi})
\right|
}\rightarrow \infty\ (\ell(\in \mathcal{L})\to\infty),
\end{equation*}
from which we have a contradiction.
Consequently, item~\eqref{prop:0614-1} holds.

Lastly, we prove item~\eqref{prop:0626-3}.
Noting $\txc^{\ell+\fr}=\tilde{x}^{\ell}+\Deltac \tilde{x}^{\ell}$ from the definition of $\widetilde{w}^{\ell+\fr}_c$
and using Proposition\,\cref{prop:order} with $F$ replaced by $h$ again, we obtain
\begin{align}
\|h(\txc^{\ell+\fr})\|_2
=&\|h(\tilde{x}^{\ell})+\nabla h(\tilde{x}^{\ell})^{\top}\Deltac \tilde{x}^{\ell}\|_2+\rO(\|\Deltac \tilde{x}^{\ell}\|_2^2)\notag\\
=&\ro(\tmu_{\ell}^{1+2\xi}),\label{al:0319-1}
\end{align}
where the second equality follows from the fact that $h(\tilde{x}^{\ell})+\nabla h(\tilde{x}^{\ell})^{\top}\Deltac \tilde{x}^{\ell}=0$ (see the second-line of \eqref{al:scalednewton-4})
and $\|\Deltac \tilde{x}^{\ell}\|_2=\rO(\tmu_{\ell}^{1+\xi})$ from \eqref{eq:mousukoshi}.
Similarly, we can derive  
\begin{equation}
\|\nabla_xL(\twc^{\ell+\fr})\|_2=\rO(\tmu_{\ell}^{2(1+\xi)})=\ro(\tmu_{\ell}^{1+2\xi}).\label{eq:0629-4}
\end{equation}
By setting $s_{\ell}=1$ for $\ell$ sufficiently large in item~\eqref{prop:0626-2} of this proposition, we have
\begin{equation}
\|G(\txl+\Deltac\txl)(\tYl+\Deltac \tYl)-\tmul I\|_{\rm F}=\rO(\tmul^{1+2\xi}).\label{eq:20200817-1}
\end{equation}
Since $\tmul^{1+2\xi}=\ro(\tmu_{\ell}^{1+2\kappa})$ holds by $\xi>\kappa>0$, 
\eqref{al:0319-1}--\eqref{eq:20200817-1} along with item~\eqref{prop:0626-1} derive the desired conclusion.
The proof is complete.
$\hfill \Box$

\subsection{Proof of Theorem\,\cref{thm:main}}\label{sec:last}
In this section, we give a proof of our main result, Theorem\,\cref{thm:main}. 
To this end, we first show the following lemma by invoking Proposition\,\cref{prop:0614} and Proposition\,\cref{prop:0626}.
We denote $\mathcal{B}_{r}(w):=\{v\in \W\mid \|v-w\|\le r\}$ for $r>0$ and $w\in \W$.
\begin{lemma}\label{thm:0710-1}
Choose $\delta\in (0,1)$ arbitrarily. Let $\P:\mathcal{W}_{++}\to \R^{m\times m}$
be a function satisfying conditions~{\rm({\bf P1})} and {\rm({\bf P2})}.
In addition, let $\xi$ and $\tau>0$ be the constants in condition~{\rm ({\bf P2})} and $\alpha$ and $\xi^{\prime}$ be arbitrary constants satisfying \eqref{al:cond}.
Then, there exists some $\overline{u}>0$ such that,
for any $\mu\in (0,\overline{u}]$,
the following properties hold if $\olw\in \mathcal{B}_{\mu^{\delta}}(w^{\ast})\cap \N_{\mu}^{\tau\mu^{1+\xi}}$.: The linear equation~\eqref{eq:pre3} with $P=\P(\olw)$ has a unique solution $\Deltap w$ and
it holds that 
\begin{align}
\olw_{\fr}:=\olw+(1-\mu^{\alpha})\Deltap w\in \mathcal{B}_{\mu_+^{\delta}}(w^{\ast})\cap \N_{\mu_+}^{\tau\mu_+^{1+\xi^{\prime}}},\label{al:0705-1}
\end{align}
where $\mu_+:=\mu^{1+\alpha}$.
Moreover, the linear equation~\eqref{eq:corr4} with $P=\P(\olw_{\fr})$ has a unique solution $\Deltac w$ and it holds that
\begin{align}
\olw_{\fr}+\Deltac w\in \mathcal{B}_{\mu_+^{\delta}}(w^{\ast})\cap \N_{\mu_+}^{\tau\mu_+^{1+\xi}}.\label{al:0705-2}
\end{align}
\end{lemma}
\begin{proof}
We first show the unique solvability of equation~\eqref{eq:pre3} by contradiction.
Suppose to the contrary that there exist sequences $\{\tmu_{\ellem}\}\subseteq \R_{++}$ and $\{\tw^{\ellem}\}\subseteq \mathcal{W}_{++}$ such that 
\begin{equation}
\lim_{\ellem\to \infty}\tmu_{\ellem}=0,\ \ \tw^{\ellem}\in \mathcal{B}_{\tmu_{\ellem}^{\delta}}(w^{\ast})\cap \N_{\tmu_{\ellem}}^{\tau\tmu_{\ellem}^{1+\xi}}
\label{eq:0802-1}
\end{equation}
but equation~\eqref{eq:pre3} with $(w,\mu,P)=(\tw^{\ellem},\tmu_{\ellem},\P(\tw^{\ellem}))$
is not uniquely solvable for each $\ellem$. By the assumptions that $\tw^{\ellem}\in \mathcal{B}_{\tmu_{\ellem}^{\delta}}(w^{\ast})$ for each $\ellem$ and $\tmu_{\ellem}\to 0$ as $\ellem\to \infty$, we obtain $\lim_{\ellem\to \infty}\tw^{\ellem}=w^{\ast}$. Thus, from Proposition\,\cref{prop:0613-2},
equation\,\eqref{eq:pre3} must be uniquely solvable for any $\ellem$ sufficiently large. However, this is a contradiction and therefore the first assertion is ensured.

Next, we show \eqref{al:0705-1}. We derive a contradiction again by supposing the existence of sequences
$\{\tmu_{\ellem}\}\subseteq \R_{++}$ and $\{\tw^{\ellem}\}\subseteq \mathcal{W}_{++}$ such that \eqref{eq:0802-1} holds but \eqref{al:0705-1} with $(\olw,\mu,\mu_+)=(\tw^{\ellem},\tmu_{\ellem},\tmu_{\ellem}^{1+\alpha})$
is invalid for any $\ellem$, namely, 
\begin{equation}
\tw^{\ellem}+(1-\tmu_{\ellem}^{\alpha})\Deltap \tw^{\ellem}\notin \mathcal{B}_{\tmu_{\ellem}^{(1+\alpha)\delta}}(w^{\ast})\cap \N_{\tmu_{\ellem}^{1+\alpha}}^{\tau\tmu_{\ellem}^{(1+\xi^{\prime})(1+\alpha)}},\label{eq:0710-1}
\end{equation}
where $\Deltap \tw^{\ellem}$ is the unique solution of equation\,\eqref{eq:pre3} with $(w,\mu,P)=(\tw^{\ellem},\tmu_{\ellem},\mathcal{P}(\tw^{\ellem}))$.
Noting that $\lim_{\ellem\to\infty}\tw^{\ellem}=w^{\ast}$ as in the former argument
and using Proposition\,\cref{prop:0614}\eqref{prop:0614-3}, we have
\begin{equation}
\tw^{\ellem}+(1-\tmu_{\ellem}^{\alpha})\Deltap \tw^{\ellem}\in \N_{\tmu_{\ellem}^{1+\alpha}}^{\tau\tmu_{\ellem}^{(1+\alpha)(1+\xi^{\prime})}} \label{eq:0711-1}
\end{equation}
for any $\ellem$ sufficiently large.
Since $\Deltap \tw^{\ellem}\to 0$ by Proposition\,\cref{prop:0613-2}\eqref{e1} and
$w^{l}\in \N_{\tmu_{l}}^{\tau\tmu_{l}^{1+\xi}}$, $\lim_{\ellem\to\infty}\tw^{\ellem}=w^{\ast}$ implies 
\begin{equation}
\lim_{\ellem\to\infty}\tw^{\ellem}+(1-\tmu_{\ellem}^{\alpha})\Deltap \tw^{\ellem}=w^{\ast}.\label{eq:0711-2}
\end{equation}
In addition, \eqref{eq:0711-1} together with \eqref{eq:0710-1} implies 
$\tw^{\ellem}+(1-\tmu_{\ellem}^{\alpha})\Deltap \tw^{\ellem}\notin \mathcal{B}_{\tmu_{\ellem}^{\delta(1+\alpha)}}(w^{\ast})$ for any $\ellem$, from which we can derive  
\begin{align*}
\tmu_{\ellem}^{\delta(1+\alpha)}&<\|\tw^{\ellem}+(1-\tmu_{\ellem}^{\alpha})\Deltap \tw^{\ellem}-w^{\ast}\|
              =\rO(\tmu_{\ellem}^{1+\alpha})
              =\ro(\tmu_{\ellem}^{\delta(1+\alpha)}),\notag
\end{align*}
where the first equality follows from \eqref{eq:0711-2} and Proposition\,\cref{prop:0219-1} with
$\{\tw^{\ell}\}$, $\{\tmu_{\ell}\}$, and $\{r_{\ell}\}$ replaced by $\{\tw^{\ellem}+(1-\tmu_{\ellem}^{\alpha})\Deltap \tw^{\ellem}\}$, $\{\tmu_{\ellem}^{1+\alpha}\}$, and $\{\tau\tmu_{\ellem}^{(1+\alpha)(1+\xi^{\prime})}\}$, respectively.                      
Therefore, we obtain an obvious contradiction $\tmu_{l}^{\delta(1+\alpha)}=o(\tmu_{\ellem}^{\delta(1+\alpha)})$ leading us to the desired relation~\eqref{al:0705-1}. Consequently, we ensure the existence of some $u_1>0$ such that \eqref{al:0705-1} together with the unique solvability of equation\,\eqref{eq:pre3} holds for $\mu\in (0,u_1]$.

In turn, we prove the second-half claim.
The unique solvability of \eqref{eq:corr4} can be proved similarly to the above arguments by replacing $(\olw,\mu)$ with $(\olw^{\fr},\mu_+)$ satisfying \eqref{al:0705-1}. So, we omit it.
We show \eqref{al:0705-2}. 
Recall
$\xi^{\prime}>\frac{\xi}{2}$
from \eqref{al:cond} and choose $\xi^{\dprime}\in \R$ such that
\begin{equation}
\xi^{\prime}>\xi^{\dprime}>\frac{\xi}{2}.\label{eq:dprime}
\end{equation}
To prove \eqref{al:0705-2}, it actually suffices to show that, for a sufficiently small $\mu$, we have 
\begin{equation}
\olw_{\fr}+\Deltac w\in \mathcal{B}_{\mu_+^{\delta}}(w^{\ast})\cap \N_{\mu_+}^{\tau\mu_+^{1+2\xi^{\dprime}}}.\label{eq:0102}
\end{equation}
Indeed,
note that $\mu_+=\mu^{1+\alpha}<1$ by taking a sufficiently small $\mu$.
Hence, the assumption that $\xi^{\dprime}>\xi/2$ yields
$\mu_+^{1+2\xi^{\dprime}}<\mu_+^{1+\xi}$ entailing $\N_{\mu_+}^{\tau\mu_+^{1+2\xi^{\dprime}}}\subseteq \N_{\mu_+}^{\tau\mu_+^{1+\xi}}$. Therefore, this fact together with \eqref{eq:0102} derives \eqref{al:0705-2}.

For the sake of proving \eqref{eq:0102},
take $\mu\in (0,u_1]$ with $u_1$ determined above. Then, \eqref{al:0705-1} holds and equation\,\eqref{eq:corr4} is uniquely solvable.
Moreover, for a contradiction, suppose that there exist sequences $\{\tw^{\ellem+\fr}\}$ and $\{\tmu_{\ellem}\}$ 
such that, by writing $p_{\ellem}:=\tmu_{\ellem}^{1+\xi}$ for each $\ellem$, it holds that
$$
\lim_{\ellem\to\infty}\tmu_{\ellem}=0,\
\tw^{\ellem+\fr}+\Deltac \tw^{\ellem+\fr}\notin \mathcal{B}_{p_{\ellem}^{\delta}}(w^{\ast})\cap\N_{p_{\ellem}}^{\tau p_{\ellem}^{1+2\xi^{\dprime}}},\
\tw^{\ellem+\fr}\in \mathcal{B}_{p_{\ellem}^{\delta}}(w^{\ast})\cap \N_{p_{\ellem}}^{\tau p_{\ellem}^{1+\xi^{\prime}}},
$$
where $\Deltac \tw^{\ellem+\fr}$ denotes the unique solution of \eqref{eq:corr4} with $(w,\mu,P)=(\tw^{\ellem+\fr},p_{\ellem},\mathcal{P}(\tw^{\ellem+\fr}))$.
A contradiction can be then derived by the same argument as the one for \eqref{al:0705-1}, where, in place of Proposition\,\cref{prop:0614}\,\eqref{prop:0614-3} and Proposition\,\cref{prop:0613-2}\,\eqref{e1} ,
we used Proposition\,\cref{prop:0626}\,\eqref{prop:0626-3} with $\{(\tw^{\ell},\Deltac \twl,\tmu_{\ell},\kappa)\}$ replaced 
by $\{(\tw^{\ellem+\fr},\Deltac \tw^{\ellem+\fr},p_{\ellem},\xi^{\dprime})\}$ and also utilize Proposition\,\cref{prop:0613-2}\,\eqref{e2}
with $\{(\tw^{\ell},\Deltac \twl,\tmu_{\ell})\}$ replaced by $\{(\tw^{\ellem+\fr},\Deltac \tw^{\ellem+\fr},p_{\ellem})\}$. 
We thus conclude \eqref{eq:0102}, and ensure the existence of some $u_2>0$ such that \eqref{eq:0102} holds and the unique solvability of equation\,\eqref{eq:corr4} is valid for $\mu\in (0,u_2]$.  

Finally, by setting $\ol{u}:=\min(u_1,u_2)$, the whole proposition is proved.
\hfill
$\Box$
\end{proof}

We are now ready to prove our main theorem.
\subsubsection*{Proof of Theorem\,\cref{thm:main}}
{Choose $\delta\in (0,1)$ arbitrarily.
By taking $w^0\in \N_{\mu_0}^{\tau \mu_0^{1+\xi}}$ with $\mu_0=\gamma \|\Xi_0^I(w^0)\|$ sufficiently close to $w^{\ast}$, 
we obtain  
$\|w^0-w^{\ast}\|<\|w^0-w^{\ast}\|^{\frac{\delta}{2}}<\mu_0^{\delta}$, where
the first inequality follows from $0<\delta<1$ and the last inequality does from Proposition\,\ref{prop:0219-1} and $\delta/2<\delta<1$.
We thus have
\begin{equation}
w^0\in \mathcal{B}_{\mu_0^{\delta}}(w^{\ast})\cap \N_{\mu_0}^{\tau \mu_0^{1+\xi}}.\label{eq:0123}
\end{equation}}
Now, let us prove item~\eqref{mThm:1} by mathematical induction.
To begin with, note that according to Lemma\,\cref{thm:0710-1}, there exists some $\overline{u}>0$ such that
\begin{equation}
\begin{array}{c}
\mu\in (0,\overline{u}];\\
w\in \mathcal{B}_{\mu^{\delta}}(w^{\ast})\cap \N_{\mu}^{\tau \mu^{1+\xi}}
\end{array}
\Longrightarrow
\begin{array}{c}
\mbox{\eqref{eq:pre3} with $P=\mathcal{P}(w)$ is u.s.};
\\
\mbox{\eqref{eq:corr4} with $(w,\mu,P)=(w_{\fr},\mu_+,\mathcal{P}(w_{\fr}))$ is u.s.};\\
\eqref{al:0705-1},\ \eqref{al:0705-2}
\end{array}
\label{eq:0115}
\end{equation}
where
``u.s.'' stands for ``uniquely solvable'', $w_{\fr}:=w+(1-\mu^{\alpha})\Deltap w$,
and $\mu_+:=\mu^{1+\alpha}$. Let us prove \eqref{eq:0704-1721} for $k=0$.
Recall that $\xi$ is the constant chosen in the initial setting of Algorithm\,\cref{alg4}.
Again, take $w^0$ so close to $w^{\ast}$ that $\mu_0=\gamma\|\Xi^I_0(w^0)\|<\min(0.95,\overline{u})$. 
Under the setting $(w,\mu)=(w^0,\mu_0)$, we obtain \eqref{al:0705-1} and \eqref{al:0705-2} from \eqref{eq:0123} and \eqref{eq:0115}. This implies that $G(x^0+(1-\mu_0^{\alpha})\Deltap x^0)\in \mS^m_{++}$ and $Y_0+(1-\mu_0^{\alpha})\Deltap Y_0\in \mS^m_{++}$, and therefore
$1-\mu_0^{\alpha}$ is set to $\bar{s}^{\rm t}_0$ in Line\,\ref{al:skt} of Algorithm\,\cref{alg4}.
Hence, $\mu_1=(1-\bar{s}^{\rm t}_0)\mu_0=\mu_0^{1+\alpha}$, i.e., \eqref{eq:0704-1721} with $k=0$.
In addition, by \eqref{eq:0115} with $(w,\mu)=(w^0,\mu_0)$ again,
the linear equation~\eqref{eq:pre3} with $(w,P,\mu)=(w^0,\P(w^0),\mu_0)$ and \eqref{eq:corr4} with
$(w,P,\mu)=(w^{\fr},\P(w^{\fr}),\mu_{\fr})$ are ensured to have unique solutions. Moreover, we obtain conditions~\eqref{eq:0702-2329} and \eqref{eq:0704-1726} for $k=0$. We thus conclude the desired conditions altogether for $k=0$.

Subsequently, suppose that conditions~\eqref{eq:0704-1721}, \eqref{eq:0702-2329}, and \eqref{eq:0704-1726} together with $\mu_k<\min(0.95,\ol{u})$ hold for some $k\ge 0$, which imply 
$w^{k+1}\in \mathcal{B}_{\mu_{k+1}^{\delta}}(w^{\ast})\cap \N_{\mu_{k+1}}^{\tau\mu_{k+1}^{1+\xi}}$ and $\mu_{k+1}<\overline{u}$, that is, the left-hand side of \eqref{eq:0115} with $w=w^k$ and $\mu=\mu_k$.
We thus, by the right-hand side of \eqref{eq:0115}, obtain conditions\,\eqref{eq:0704-1721}, \eqref{eq:0702-2329}, and \eqref{eq:0704-1726} with $k$ replaced by $k+1$, and moreover establish the unique solvability of equations\,\eqref{eq:pre3} and \eqref{eq:corr4} at the $(k+1)$-th iteration.
By induction, we conclude item~\eqref{mThm:1} for each $k\ge 0$.

We next prove item~\eqref{mThm:2}. By the above proof, we see $w^{k}\in \mathcal{B}_{\mu_k^{\delta}}(w^{\ast})\cap N_{\mu_{k}}^{\tau\mu_{k}^{1+\xi}}$ for each $k\ge 0$ and ensure that $\{w^k\}$ converges to $w^{\ast}$ with fulfilling $w^k\in N_{\mu_{k}}^{\tau\mu_{k}^{1+\xi}}$. Hence, the assertion readily follows because Proposition\,\cref{prop:0219-1} together with \eqref{eq:0704-1721} yields 
\begin{equation}
\|w^{k+1}-w^{\ast}\|=\rO(\mu_{k+1})=\rO(\mu_{k}^{1+\alpha})=\rO(\|w^k-w^{\ast}\|^{1+\alpha}).\notag 
\end{equation}
Lastly, by \eqref{al:cond}, $1+\alpha$ is bounded from above as 
$
1+\alpha<\frac{2(1+\xi)}{2+\xi}<\frac{4}{3}.
$
The proof is complete.
$ \hfill \Box $

\section{Numerical experiments}\label{sec:num}
In this section, we conduct numerical experiments to compare the efficiency of members of the MT family. 
All the programs are implemented with C++ using Eigen-3.3.9 (\url{http://eigen.tuxfamily.org/dox/index.html}) as linear algebra software, and run on a machine with Intel(R) Xeon(R) CPU E5-1620 v3@3.50GHz and 10.24GB RAM.

\subsection{Experimental setting}
As the parameters in Algorithm\,\ref{alg4}, we choose  
$$
\alpha = 0.3,\ \beta = 0.5.
$$
The starting points will be described at section\,\ref{sec:testp}.
To make Algorithm\,\ref{alg4} posses the global convergence property, we perform the following procedure
just before Line~\ref{linescaling} in Algorithm\,\cref{alg4}:
\begin{quote} 
``If $w^k\notin \mathcal{N}^{r_k}_{\mu_k}$, then find $v\in \mathcal{N}^{r_k}_{\mu_k}$ and replace $w^k$ by $v$'',
\end{quote}
where we set $r_k:=\tau \mu_k^{1+\xi}$ at each $k$.
The parameters are selected as 
$$
\xi = \frac{13}{14},\ \tau = 2.
$$	
The above $(\alpha,\xi)$ together with $\xi^{\prime} = \frac{215}{728}$ fulfills condition\,\eqref{al:cond}.
{In order to find $v\in \mathcal{N}^{r_k}_{\mu_k}$, we implement the following Newton method, which
produces 
$\{v^l\}\subseteq \W_{++}$ by 
$$
v^{l+1} \leftarrow v^l + \bar{s}_l\Delta v,\ \ \bar{s}_l\in (0,1], 
$$ 
in which $\Delta v:=(\Delta x,\Delta Y,\Delta z)$ is a solution to the Newton equation $\mathcal{J}\Xi^P_{\muk}(v^l)\Delta v=-\Xi^P_{\muk}(v^l)$. 
Choices of search directions in the MT family are described later.
Let $\Psi_{\mu}^P(w):=\fr\|\Xi_{\mu}^P(w)\|^2$.
To determine the step-size $\bar{s}_l$, we perform the backtracking line search using $\Psi_{\muk}^I(w)$
as merit function: We find the smallest integer $\bar{\ell}\ge 0$ satisfying
\begin{equation}
\Psi_{\muk}^I(v^l+\beta_2(\bar{\ell})\Delta v)\le \left(1 - 0.25 \cdot 
\beta_2(\bar{\ell})\right) \Psi_{\muk}^I(v^l),\label{eq:linesearch}
\end{equation}
where $\beta_2(\bar{\ell}) := 0.99^{\bar{\ell}} \bar{s}$ and $\bar{s}$ is computed by \eqref{eq:sxsY} because linear semidefinite constraints are dealt with in this experiment.
Such $\bar{\ell}$ exists when the above-mentioned scaled Newton equation is uniquely solvable.
\footnote{
Since $\mathcal{J}\Xi^P_{\muk}(v^l)\Delta v=-\Xi^P_{\muk}(v^l)$ is assumed to be nonsingular,
there exists a nonnegative integer $\bar{\ell}$ such that
$\Psi_{\muk}^P(v^l+\beta_2(\bar{\ell})\Delta v)\le \left(1 - 0.25 \cdot 
\beta_2(\bar{\ell})\right) \Psi_{\muk}^P(v^l)$, which is rewritten as  
\eqref{eq:linesearch} because of $\Psi_{\muk}^P(w)=\Psi_{\muk}^I(w)$.
}   
Then, we set $\bar{s}_l\leftarrow \beta_2(\bar{\ell})$.}
For the sake of practical implementation, we terminate the above line search procedure if $\beta_2(\bar{\ell})\le 10^{-8}$ or the size of the difference of both the sides in \eqref{eq:linesearch} is not greater than $10^{-9}$. If it is stopped due to $\beta_2(\bar{\ell})\le 10^{-8}$, we set $\bar{s}_l\leftarrow \bar{s}$.
We call the above procedure the Inner-Newton procedure. We solve the Newton equations of form\,\eqref{al:20230416-1} and 
\eqref{al:20230416-2}.

Algorithm\,\ref{alg4} terminates once any of the following conditions is satisfied:
(a) $\|\Phi^1_0(w)\|\le 10^{-7}$ (thus $\|\Xi^I_0(w)\|\le \|\Phi^1_0(w)\|\le 10^{-7}$ from Proposition\,\ref{prop:xyx}); (b) $\mu_k\le 10^{-9}$; (c) the number of iterations within the Inner-Newton procedure passes 100 and 
its final solution is not contained by $\mathcal{N}^{10\mu_k}_{\mu_k}$; 
(d) the number of outer iterations of Algorithm\,\ref{alg4} passes 50; (e) running time exceeds 300 seconds; (f) the Newton equation is not solvable.

\subsection{Test problems and numerical results}\label{sec:testp}
We solve the following synthetic nonconvex NSDP:
\begin{align}\label{al:pro1}
\begin{array}{rcl}
\displaystyle\mathop{\rm Minimize}_{x\in \R^n}& &c^{\top}x+\frac{1}{2}x^{\top}Mx - a_1 \|x\|_2^3\\
               	               \mbox{subject~to}& &G(x):=a_2I + \sum_{i=1}^nx_iF_i\in \mS^m_+,\\
               	                                        & & h(x) := \sum_{i=1}^n(x_i-a_3)^2 - n\cdot a_3^2 = 0,
\end{array}
\end{align}
where $F_i\in \mS^m$ for $i=1,2,\ldots,n$, $c\in \R^n$, $a_1,a_2,a_3>0$, and $M\in \mS^n$.
Note that the origin $0$ is feasible to the above NSDP, in particular, $G(0)=a_2I\in \mS^m_{++}$ holds.  Moreover, the NSDP has at least one global optimum, since
the feasible region of \eqref{al:pro1} is compact due to the constraint $h(x)=0$.

Each component of $c, M, F_i\ (i=1,\ldots,n)$ is chosen from $[-2,2]$, 
$a_1,a_2$ from $[0,1]$, and $a_3$ from $[0.1,2.1]$, randomly following the uniform distribution.
We generate 50 instances for each of $(m,n)=(50,25),(50,50),(50,100),(100,100)$ in this manner. 
We set $(x^0,Y_0,z^0)\leftarrow (0,I,0)$ and $\mu_0\leftarrow 1$ as the starting point and the initial barrier parameter of the whole algorithm.

We summarize the obtained results in Table\,\ref{tab}, where 
{\tims} stands for the averaged running time in seconds, 
{\outite} for the averaged number of iterations for Algorithm~\ref{alg4}, and 
{\innite} for the averaged total number of iterations in the Newton method performed for computing $v\in \mathcal{N}^{r_k}_{\mu_k}$ for each $k$.  
Moreover, {\resKKT} and {\resKKTMT} represent the averaged values of $\|\Phi^1_0(\cdot)\|$ and $\|\Xi_0^I(\cdot)\|$ at output solutions, respectively, and {\fail} stands for the percentage of successfully solved instances out of the 50 ones.
``MT'', ``NT'', ``{\HKM}'', ``{\HKM}-dual'', and ``MTW'' represent search directions which are used in Algorithm\,\ref{alg4} for solving each instance.
The best record is marked in bold in each column.

From {\fail}, we observe that the MT and NT solved all the problems successfully. Meanwhile, {\HKM}, {\HKM}-dual and MTW failed for some instances, although they gained solutions with $\|\Phi_0^1\|\le 10^{-6}$. From time(s), we see that 
the NT and {\HKM} tend to spend the least time, while the MT the most.
This is mainly due to the computation of $\mathcal{L}_{G(x)}^{-1}(\G_j)$ in constructing the matrix $\mathcal{B}(x,Y)$ in \eqref{al:20230416-1}.
Moreover, unlike the other directions, the MT needs to compute all the elements of $\mathcal{B}(x,Y)$ which is not necessarily symmetric. 
According to {\outite} and {\innite}, the MT compute solutions with a relatively less number of iterations among all directions.
Indeed, it attains the smallest {\outite} and {\innite} for $(m,n)=(100,100),(25,50)$, and the second smallest for $(m,n)=(50,50),(50,100)$.
These results indicate the MT spends more CPU-time than the other directions, but works stably for solving the NSDP.

\begin{table}[htbp]
\caption{Averaged results about 50 instances of problem\,\eqref{al:pro1}}
\centering
\begin{tabular}{c|c|cccccc}
$(m,n)$ & direction & {\tims} & {\outite} & {\innite}& {\resKKT} &{\resKKTMT}& {\fail} \\   \hline
&MT&0.63&{\bf 27.7}&{\bf 4.2}&3.90e-08&3.68e-08 & 100 \\  
&NT&0.28&27.9&4.6&3.62e-08&3.47e-08&100\\  
(25,50)&{\HKM}&{\bf 0.26}&28.0&4.84&4.35e-08&4.21e-08&100\\  
&{\HKM}-dual&0.40&27.8&4.4&{\bf 3.48e-08}&{\bf 3.33e-08}&100\\  
&MTW&0.41&28&4.8&4.17e-08&4.04e-08&100\\ \hline %
&MT&3.12&29.5&5.52&4.15e-08&3.98e-08 & 100 \\  
&NT&1.49&29.8&6.2&{\bf 4.03e-08}&{\bf 3.84e-08}&100\\  
(50,50)&{\HKM}&{\bf 1.36}&30.1&6.84&4.39e-08&4.20e-08&100\\  
&{\HKM}-dual&2.13&{\bf 29.4}&{\bf 5.48}&4.23e-08&3.87e-08&100\\  
  &MTW&2.18&29.9&6.48&4.30e-08&4.12e-08&100\\ \hline %
  &MT&14.47&31.52&6.44&4.39e-08&4.23e-08 & 100 \\  
&NT&{\bf 7.38}&31.98&7.56&{\bf 4.00e-08}&{\bf 3.71e-08}&100\\  
(100,50)&{\HKM}&8.28&33.22&23.02&5.15e-08&5.06e-08&100\\  
&{\HKM}-dual&11.07&{\bf 31.12}&{\bf 5.84}&4.17e-08&3.93e-08&98\\  
&MTW&11.89&32.58&9.6&4.18e-08&3.91e-08&100\\ \hline %
  &MT&33.88&{\bf 30.70}&{\bf 5.8}&4.16e-08&3.75e-08 & 100 \\  
&NT&{13.56}&31.44&7.6&{\bf 4.14e-08}&{\bf 3.69e-08}&100\\  
(100,100)&{\HKM}&{\bf 13.27}&32.72&13.15&4.70e-08&4.34e-08&94\\  
&{\HKM}-dual&20.29&30.74&6.16&4.40e-08&3.96e-08&100\\  
&MTW&21.40&31.81&8.55&4.52e-08&4.04e-08&94\\ \hline %

\end{tabular}\label{tab}
\end{table}

\section{Concluding remarks}\label{sec:con}
In this paper, we have considered nonlinear semidefinite optimization problems (NSDPs) and studied the primal-dual interior point method (PDIPM) for NSDPs using the Monteiro-Tsuchiya (MT) family of directions
generated by applying Newton's method to 
$G(x)^{\fr}YG(x)^{\fr}=\mu I$ after scaling $G(x)$ and $Y$. 
We also have analyzed its local superlinear convergence to a KKT triplet under some assumptions.
The analysis in the paper is specific to the MT family, and quite different from those for the PDIPMs using the Monteiro-Zhang family.  
We can consider a different family of search directions by interchanging the roles of $G(x)$ and $Y$ above.
The properties of the MT family that have been proven can be extended to this family easily.
Finally, a possible future work is to develop globally convergent PDIPMs using the MT family. Moreover, it would be crucial to investigate theoretical advantages and disadvantages of each direction in the MT family.

{\vspace{0.5em}\noindent{\bf Data Availability Statement:} All data in the paper are available from the corresponding author on reasonable request. There is no conflict of interest in writing the paper.

\vspace{0.5em}\noindent{\bf Acknowledgments}: The author thanks Professor Yoshiko Ikebe, Professor Mirai Tanaka, and Professor Makoto Yamashita for numerous comments and suggestions.
He is also sincerely grateful for the anonymous reviewers for many crucial suggestions. }

\bibliography{ref}
\bibliographystyle{spmpsci}      

\label{sec:appendix1}
\def\thesection{A}
\renewcommand{\theequation}{A.\arabic{equation}}
\renewcommand{\thetheorem}{A.\arabic{theorem}}
\setcounter{equation}{0}
\setcounter{theorem}{0}
\section{Omitted Proofs}\label{sec:appendix1}
\subsection{Proof of Proposition\,\cref{prop:scaling}}
Before the proof of Proposition\,\cref{prop:scaling}, we first give two propositions. 
Choose arbitrary sequences $\{\tw^{\ell}\}$ and $\{\tmu^{\ell}\}$ satisfying \eqref{eq:P2} in Condition~({\bf P2}).
To show the proposition, we prepare the following two claims.

Since the semidefinite complementarity condition that $G(x^{\ast})\bullet Y_{\ast}=0$, $G(x^{\ast})\in \mS^m_+$, and $Y_{\ast}\in \mS^m_+$ holds, 
the matrices $G(x^{\ast})$ and $Y_{\ast}$ can be simultaneously diagonalized, namely, there
 exists an orthogonal matrix $Q_{\ast}\in \R^{m\times m}$ such that
 \begin{equation}
 G(x^{\ast})=Q_{\ast}\begin{bmatrix}\Lambda_1&O\\O&O
 \end{bmatrix}Q_{\ast}^{\top},\ Y_{\ast}=Q_{\ast}
 \begin{bmatrix}O&O
 \\
 O&\Lambda_2
 \end{bmatrix}
 Q^{\top}_{\ast},
\label{eq:mat}
 \end{equation}
 where $\Lambda_1\in \R^{r_{\ast}\times r_{\ast}}$ with $r_{\ast}:={\rm rank}\,G(x^{\ast})$
 is a positive diagonal matrix and $\Lambda_2\in \R^{(m-r_{\ast})\times (m-r_{\ast})}$ is a nonnegative diagonal matrix.
 The diagonal entries of $\Lambda_1$ and $\Lambda_2$ are the eigenvalues of $G(x^{\ast})$ and $Y_{\ast}$, respectively.
The following proposition is obtained from \cite[Lemma~3]{yamashita2012local} under the assumption that $\tw^{\ell}\in \N_{\tmul}^{r_{\ell}}$ with $r_{\ell}={\rm o}(\tmul)$.
\begin{PropA}\label{appendix:prop:0219-1}
It holds that 
$$
G_{\ell} =\begin{bmatrix}
{\rm \Theta}(1)& {\rm O}(\tmul)\\
{\rm O}(\tmul)&{\rm \Theta}(\tmul)
\end{bmatrix},\ \tY_{\ell}
=\begin{bmatrix}
{\rm \Theta}(\tmul) & {\rm O}(\tmul)\\
{\rm O}(\tmul)&{{\rm \Theta}}(1) 
\end{bmatrix},
$$
where both the matrices are partitioned into the four blocks with the same sizes as those in \eqref{eq:mat}. The above expressions indicate upper-bounds of the magnitude of the block matrices. For example, $\|\mbox{the $(1,1)$-block of $\Gl$}\|_{\rm F}={\rm \Theta}(1)$.
Moreover, the sequences of the inverse matrices satisfy 
$$
\Gl^{-1}  =\begin{bmatrix}
{\rm \Theta}(1)& {\rm O}(1)\\
{\rm O}(1)&{\rm \Theta}(\tmul^{-1})
\end{bmatrix},\ \tYl^{-1}=\begin{bmatrix}
{\rm \Theta}(\tmul^{-1}) & {\rm O}(1)\\
{\rm O}(1)&{\rm \Theta}(1) 
\end{bmatrix}.
$$
\end{PropA}
The next one will be used to prove Case~(i) of Proposition\,\cref{prop:scaling}.
\begin{PropA}\label{prop:AA}
Let $U:=\mathcal{L}_{X^{\fr}}^{-1}(\Delta X)$ for $X\in \mS^m_{++}$ and $\Delta X\in \mS^m$.
Then,
$$
\|UX^{-\fr}\|_{\rm F}=\|X^{-\fr}U\|_{\rm F}\le \frac{1}{\sqrt{2}}\|X^{-\fr}\Delta XX^{-\fr}\|_{\rm F}
\le \sqrt{\frac{m}{2}}\|\Delta X\|_{\rm F}\|X^{-1}\|_{\rm F}$$
\end{PropA}

\begin{proof}
Since the first equality is obvious, we show the inequalities part.
From $U=\mathcal{L}_{X^{\fr}}^{-1}(\Delta X)$ 
it follows that 
$UX^{\fr}+X^{\fr}U=\Delta X$, which implies
\begin{equation}
UX^{-\fr}+X^{-\fr}U=X^{-\fr}\Delta XX^{-\fr}. \label{eq:0217-2331}
\end{equation}
Then, the first desired inequality follows from \cite[Lemma~2.1]{monteiro1999polynomial}.
The second one is obtained from 
$$
\frac{
\|X^{-\fr}\Delta XX^{-\fr}\|_{\rm F}
}{\sqrt{2}}
\le 
\frac{
\|X^{-\fr}\|_{\rm F}^2\|\Delta X\|_{\rm F}
}{\sqrt{2}}
\le 
\sqrt{\frac{m}{2}} \|X^{-1}\|_{\rm F}\|\Delta X\|_{\rm F},
$$
where the second inequality follows from  $\|X^{-\fr}\|_{\rm F}^2={\rm Tr}(X^{-1})\le \|I\|_{\rm F}\|X^{-1}\|_{\rm F}=\sqrt{m}\|X^{-1}\|_{\rm F}$.
Hence, the proof is complete.
\hfill$\Box$
\end{proof}

Let us start proving Proposition\,\cref{prop:scaling}.
For each $\ell$, let 
\begin{align*}
\tPl:=\mathcal{P}(\twl),\ \hGl:=\tPl\Gl\tPl^{\top},\ 
\hYl:=\tPl^{-\top}\tYl\tPl^{-1}.
\end{align*}
For other notations such as $\hGi$, see Condition~($\bf P2$).
Note that for $w\in \mathcal{W}$ and $\mu>0$,
\begin{equation}
\|\hG(x)^{\fr}\hY\hG(x)^{\fr}-\mu I\|_{\rm F}
=\|G(x)^{\fr}YG(x)^{\fr}-\mu I\|_{\rm F}
\le \|G(x)Y-\mu I\|_{\rm F},
\label{eq:0117-0}
\end{equation}
where the equality is easily verified by comparing the squares of both the sides
and the inequality follows from Proposition\,\cref{prop:xyx} with $X=G(x)$. Combining the above relation with $\|\Gl\tYl-\tmul I\|_{\rm F}=\rO(\tmul^{1+\xi})$, we have
\begin{equation}
\|\hGl^{\fr}\hYl\hGl^{\fr}-\tmul I\|_{\rm F}=\rO(\tmul^{1+\xi}).
\label{eq:0117-1}
\end{equation}
We often use the following equations: 
\begin{align}
&\|\Gl\|_{\rm F}=\rO(1),\ \|\tYl\|_{\rm F}=\rO(1),\ \|\Gi(\txl)\|_{\rm F}
=\rO(1),\label{eq;0118-2}\\
&\tmul\|\Gl^{-1}\|_{\rm F}=\rO(1),\ \tmul\|\tYl^{-1}\|_{\rm F}=\rO(1),\label{eq;0118-1}
\end{align}
where the equations in \eqref{eq;0118-2} are derived from
$\lim_{\ell\to\infty}(\tx^{\ell},\tYl)=(x^{\ast},Y_{\ast})$ and the continuity of $\Gi$ and $G$, and those in \eqref{eq;0118-1} follow from Proposition\,\cref{appendix:prop:0219-1}.

Now, we proceed to the proof of Cases\,(i)-(v).
Fix $i\in \{1,2,\ldots,n\}$ arbitrarily and write
$\Uls:=\mathcal{L}^{-1}_{\hGl^{\fr}}(\hGi(\txl)).$
We first show Case~(i) with $\tPl=I$ for any $\ell$. 
Note $\zli={\tmul}\Uls\Gl^{-\fr}$ with $\Uls=\mathcal{L}^{-1}_{\Gl^{\fr}}(\Gi(\txl))$ in this case.
We then have
$\tmul \Uls\Gl^{\fr}+\tmul \Gl^{\fr}\Uls=\tmul \Gi(\txl)$,
which together with Proposition\,\cref{prop:AA} with $(X,\Delta X)=(\Gl,\Gi(\txl))$ implies 
\begin{align*}
\|\zli\|_{\rm F}=
\tmul\|\Uls\Gl^{-\fr}\|_{\rm F}\le 
\tmul\sqrt{\frac{m}{2}}\|\Gl^{-1}\|_{\rm F}\|\Gi(\txl)\|_{\rm F}=\rO(1),
\end{align*}
where we used \eqref{eq;0118-2} and \eqref{eq;0118-1}.
Hence, $\{\zli\}$ is bounded.
We next show Case~(ii) with $\tPl=\Gl^{-\fr}$.
By $\hGl=I$ for each $\ell$, we have $\Uls=\hGi(\txl)/2$, which together with \eqref{eq;0118-2} and \eqref{eq;0118-1} implies
$$
\|\zli\|_{\rm F}=\tmul \|\Gl^{\fr}\Uls\Gl^{-\fr}\|_{\rm F}=\frac{\tmul}{2} \|\Gi(\txl)\Gl^{-1}\|_{\rm F}=\rO(1).
$$
Thus, $\{\zli\}$ is bounded for Case~(ii).

In what follows, we show the remaining cases in a unified manner. 
As will be shown later, in each of Cases~(iii)-(v), there exists some $\rast>0$ such that  
\begin{equation}
\|\hGl^{\frac{1}{\rast}}-\tmul I\|_{\rm F}=\rO(\tmul^{1+\xi}).\label{eq:0119-0106}
\end{equation}

Let $\Uellast :=\frac{\tmul^{-\frac{\rast}{2}}}{2}\hGi(\txl)$ for each $\ell$.
The expression 
$\tmul\left\|\tPl^{-1}\Uellast\hGl^{-\fr}\tPl\right\|_{\rm F}$ is evaluated as 
\begin{equation}
  \tmul\left\|\tPl^{-1}\Uellast\hGl^{-\fr}\tPl\right\|_{\rm F}=\frac{\tmul^{1-\frac{\rast}{2}}\|\Gil\tPl^{\top}\hGl^{-\fr}\tPl\|_{\rm F}}{2}
=\rO(\tmul^{1-\rast}\|\tPl\|_{\rm F}^2),\label{eq:0119}
\end{equation}
where
the first equality follows from
$\widehat{\Gi}(\txl)=\tPl\Gi(\txl)\tPl^{\top}$
and the last one from $\|\hGl^{-\frac{1}{2}}\|_{\rm F}=\rO(\tmul^{-\frac{\rast}{2}})$ by
\eqref{eq:0119-0106} and $\|\Gi(\txl)\|_{\rm F}=\rO(1)$ as in \eqref{eq;0118-2}.
Furthermore, let
\begin{equation}\label{eq:hgl}
\hGl=Q_{\ell}D_{\ell}Q_{\ell}^{\top}
\end{equation}
be an eigen-decomposition of $\hGl$ with an appropriate orthogonal matrix $Q_{\ell}\in \R^{m\times m}$ and a diagonal matrix $D_{\ell}\in \R^{m\times m}$ with the eigenvalues of $\hGl$ aligned on the diagonal.
Notice that $\hGl^{\fr}=Q_{\ell}D_{\ell}^{\fr}Q_{\ell}^{\top}$.
Denote $\dd_{p,\ell}:=(D_{\ell})_{pp}\in \R$ for each $p=1,2,\ldots,m$ and $\UQl:=Q_{\ell}^{\top}\Uls Q_{\ell}$. 
By multiplying
$\Ql^{\top}$ and $\Ql$ on both sides of  
$\Uls\hGl^{\fr}+\hGl^{\fr}\Uls=\hGi(\txl)$, it follows from \eqref{eq:hgl} that 
$\UQl D_{\ell}^{\fr}+D_{\ell}^{\fr}\UQl=\Ql^{\top}\hGi(\txl)\Ql$, which together with $\dd_{p,\ell}=(D_{\ell})_{pp}$ for each $p,\ell$
yields
\begin{equation}
(\UQl)_{p,q}=\frac{(\Ql^{\top}\hGi(\txl) \Ql)_{p,q}}{\dd_{p,\ell}^{\fr}+\dd_{q,\ell}^{\fr}}\label{eq:UQ}
\end{equation}
for each $p,q=1,2,\ldots,m$.
From \eqref{eq:0119-0106}, 
for each $p=1,2,\ldots,m$, there exists some $\{\delta_{p,\l}\}\subseteq \R$ satisfying $\delta_{p,\l}=\rO(\tmul^{1+\xi})$ and 
$
\dd_{p,\l}=(\tmul+\delta_{p,\l})^{\rast}.
$
By taking the fact of $\tmul>0$ into account, for each $p$, the mean-value theorem implies that for some $\bar{s}_{p,\l}\in [0,1]$
\begin{align}
\dd_{p,\l}^{\fr}-\tmul^{\frac{\rast}{2}}=(\tmul+\delta_{p,\l})^{\frac{\rast}{2}}-\tmul^{\frac{\rast}{2}}
                                   ={\frac{1}{2}\rast\delta_{p,\l}(\tmul+\bar{s}_{p,\l}\delta_{p,\l})^{\frac{\rast}{2}-1}}.\label{al:0815-1}
\end{align}
Notice that $\tmul+\bar{s}_{p,\l}\delta_{p,\l}=\rTheta(\tmul)$ and $\dd_{p,\l} =\rTheta(\tmul^{\rast})$ for all $p$.
Then, for each $p,q=1,2,\ldots,m$, \eqref{al:0815-1} yields 
\begin{align}
\frac{1}{2\tmul^{\frac{\rast}{2}}}-\frac{1}{\dd_{p,\l}^{\fr}+\dd_{q,\l}^{\fr}}
&=\frac{\dd_{p,\l}^{\fr}+\dd_{q,\l}^{\fr}-2\tmul^{\frac{\rast}{2}}}{2(\dd_{p,\l}^{\fr}+\dd_{q,\l}^{\fr})\mu^{\frac{\rast}{2}}}\notag \\
&=\rO\left(
\frac{
\rast\delta_{p,\l}(\tmul+\bar{s}_{p,\l}\delta_{p,\l})^{\frac{\rast}{2}-1}+
\rast\delta_{q,\l}(\tmul+\bar{s}_{q,\l}\delta_{q,\l})^{\frac{\rast}{2}-1}
}
{
4\tmul^{\rast}
}
\right)\notag \\
&=\rO(\tmul^{\xi - \frac{\rast}{2}}),
\end{align}
where the last equality follows from $\delta_{p,\l}=\rO(\tmul^{1+\xi})$ and $\delta_{q,\l}=\rO(\tmul^{1+\xi})$.
From this fact together with $\|\Ql^{\top}\hGi(\txl)\Ql\|_{\rm F}=
\|\Ql^{\top}\tPl\Gi(\txl)\tPl^{\top}\Ql\|_{\rm F}=
\rO(\|\tPl\|_{\rm F}^2)$, we obtain, for each $p,q$,
\begin{equation*}
(\UQl-\Ql^{\top}\Uellast \Ql)_{p,q}={(\Ql^{\top}\hGi(\txl)\Ql)_{p,q}}\left(\frac{1}{\dd_{p,\l}^{\fr}+\dd_{q,\l}^{\fr}}-\frac{1}{2\tmul^{\frac{\rast}{2}}}\right)\\
=\rO\left(\|\tPl\|_{\rm F}^2\tmul^{\xi - \frac{\rast}{2}}\right), 
\end{equation*}
which together with $\|\Gl^{-\fr}\|_{\rm F}=\rO(\tmul^{-\frac{\rast}{2}})$ from \eqref{eq:0119-0106} implies
\begin{align}
\|\tmul \tPl^{-1}\Ql(\UQl-\Ql^{\top}\Uellast \Ql)\Ql^{\top}\hGl^{-\fr}\tPl\|_{\rm F}&\le \tmul \|\tPl^{-1}\|_{\rm F}\|\Ql\|_{\rm F}^2\|\tPl\|_{\rm F}\|\UQl-\Ql^{\top}\Uellast \Ql\|_{\rm F}\|\hGl^{-\fr}\|_{\rm F}\notag\\
&=\displaystyle{\rO(\tmul^{\xi+1-\rast}\|\tPl\|_{\rm F}^3\|\tPl^{-1}\|_{\rm F})},\notag
\end{align}
where we used the fact that $\|\Ql\|_{\rm F}=\sqrt{m}$ because $\Ql$ is an orthogonal matrix.
Hence, by recalling $\zli=\tmul \tPl^{-1}\Uls\hGl^{-\fr}\tPl$ and using \eqref{eq:0119} we obtain
\begin{align}
\|\zli\|_{\rm F}&\le \|\tmul \tPl^{-1}\Uellast \hGl^{-\fr}\tPl\|_{\rm F}
+\|\tmul \tPl^{-1}\Ql(\UQl-\Ql^{\top}\Uellast \Ql)\Ql^{\top}\hGl^{-\fr}\tPl\|_{\rm F}\notag\\
&=\rO(\tmul^{1-\rast}\|\tPl\|_{\rm F}^2+\tmul^{\xi+1-\rast}\|\tPl\|_{\rm F}^3\|\tPl^{-1}\|_{\rm F}).\label{01119-2}
\end{align}
Hereafter, for each of Cases~(iii)-(v), we evaluate $\|\tPl\|_{\rm F}$, $\|\tPl^{-1}\|_{\rm F}$, and $\rast$, and prove the boundedness of $\{\zli\}$ by showing that the rightmost hand expression in \eqref{01119-2} is $\rO(1)$.\vspace{0.5em}\\
\noindent{\bf Case~(iii)}: 
Since $\hYl=I$, $\tPl=\tYl^{\fr}$ for each $\ell$, \eqref{eq:0117-1} implies $\|\hGl-\tmul I\|_{\rm F}=\|
\hGl^{\fr}\hYl\hGl^{\fr}-\tmul I\|_{\rm F}=\rO(\tmul^{1+\xi})$, which indicates $\rast=1$ (see \eqref{eq:0119-0106}).
Furthermore, $\|\tPl\|_{\rm F}=\|\tYl^{\fr}\|_{\rm F}=\rO(1)$ and $\|\tPl^{-1}\|_{\rm F}=\|\tYl^{-\fr}\|_{\rm F}=\rO(\tmul^{-\fr})$ by \eqref{eq;0118-1}.
Combined with \eqref{01119-2} and the assumption $\xi\ge \fr$, these results yield
$\|\zli\|_{\rm F}=\displaystyle{\rO(1+\tmul^{\xi-\fr})}=\rO(1)$.

\vspace{0.5em}
\noindent{\bf Case~(iv)}: Since $\tPl=(\tYl\Gl\tYl)^{\fr}$ and $\hGl^{-\fr}=\hYl$, we have, from \eqref{eq:0117-1},
$
\|\hGl^{\fr}-\tmul I\|_{\rm F}=\rO(\tmul^{1+\xi})
$
yielding $\rast=2$. 
Moreover, by \eqref{eq;0118-2} and $\|\Gl\tYl-\tmul I\|_{\rm F}=\rO(\tmul^{1+\xi})$,  
\begin{align*}
&\|\tPl\|_{\rm F}^2={\rm Tr}(\tYl \Gl\tYl)\le \|\tYl\|_{\rm F}\|\Gl\tYl-\tmul I\|_{\rm F}
+\tmul \|\tYl\|_{\rm F}=\rO(\tmul),\ \mbox{and}\\
&\|\tPl^{-1}\|_{\rm F}^2={\rm Tr}(\tYl^{-1}\Gl^{-1}\tYl^{-1})
                ={\rm Tr}(\Gl \Kl^2)     
                 \le \|\Gl\|_{\rm F}\|\Kl\|_{\rm F}^2
                 =\rO(\tmul^{-2}),
\end{align*}
where $\Kl:=\Gl^{-\fr}\tYl^{-1}\Gl^{-\fr}$ and
we used $\|\Kl\|_{\rm F}=\rO(\tmul^{-1})$ from \eqref{eq:0117-1} to derive the last equality. 
These results combined with \eqref{01119-2}, $\rast=2$, and
$\xi\ge \fr$ yield
$\|\zli\|_{\rm F}
=\displaystyle{\rO(1+\tmul^{\xi-\fr})}=\rO(1).$\\
\vspace{0.5em}
\noindent{\bf Case~(v)}: By $\hGl=\hYl$ and \eqref{eq:0117-1}, 
we have
$\|\hGl^2-\tmul I\|_{\rm F}=\rO(\tmul^{1+\xi})$ yielding $\rast=\fr$.
Recall that the MTW scaling matrix $\Wl$ is defined by
$
\Wl:=\Gl^{\fr}\left(\Gl^{\fr}\tYl\Gl^{\fr}\right)^{-\fr}\Gl^{\fr}
$
for each $\ell$.
Note that $\|\Gl\|_{\rm F}=\rO(1)$
and $\|\Gl^{\fr}\tYl \Gl^{\fr}\|_{\rm F}=\rTheta(\tmul)$
follow from \eqref{eq;0118-2} and \eqref{eq:0117-1}, respectively.
The first equality in \eqref{eq:0117-0} then implies
\begin{equation}
\|\Wl\|^2_{\rm F} =\|\Gl^{\fr}\left(\Gl^{-\fr}\tYl^{-1}\Gl^{-\fr}\right)^{\fr}\Gl^{\fr}\|^2_{\rm F}
           \le m\|\Gl\|_{\rm F}^2\|\Gl^{-\fr}\tYl^{-1}\Gl^{-\fr}\|_{\rm F}   
           =\rO(\tmul^{-1}), \notag
\end{equation}
which entails
$\|\Wl\|_{\rm F}=\rO(\tmul^{-\fr})$. 
Using $\|\Gl^{\fr}\tYl \Gl^{\fr}\|_{\rm F}=\rTheta(\tmul)$ again and $\|\Gl^{-1}\|_{\rm F}=\rO(\tmul^{-1})$ from \eqref{eq;0118-1}, we have 
  $\|\Wl^{-1}\|_{\rm F} = \|\Gl^{-\fr}\left(\Gl^{\fr}\tYl\Gl^{\fr}\right)^{\fr}\Gl^{-\fr}\|_{\rm F}
  =\rO(\tmul^{-\fr}).$
Hence, we obtain that 
$
\|\tPl\|_{\rm F}^2={\rm Tr}(\Wl^{-1})=\rO(\tmul^{-\fr})
$ and 
$
\|\tPl^{-1}\|_{\rm F}^2={\rm Tr}(\Wl)=\rO(\tmul^{-\fr})$.
Therefore, 
$
\|\tPl\|_{\rm F}=\rO(\tmul^{-\frac{1}{4}}),\ \|\tPl^{-1}\|_{\rm F}=\rO(\tmul^{-\frac{1}{4}}).$ 
These results combined with \eqref{01119-2}, $\rast=\fr$, and $\xi\ge \fr$ yield
$\|\zli\|_{\rm F}=\rO(1+\tmul^{\xi-\fr})=\rO(1)$.
We complete the proof.

\hfill $\square$

\end{document}